\documentclass[12pt]{amsart}
\usepackage[utf8]{inputenc}
\usepackage{amsmath}
\usepackage{amsfonts}
\usepackage{amssymb}
\usepackage{amsthm}
\usepackage{bbm}
\usepackage{tikz,graphicx,color,mathrsfs,hyperref,color}
\usepackage{caption}
\theoremstyle{plain}
\newtheorem{theorem}{Theorem}
\newtheorem*{hypothesisF}{Hypothesis $\mathcal{F}$}
\newtheorem{lemma}[theorem]{Lemma}
\newtheorem{corollary}[theorem]{Corollary}
\newtheorem{proposition}[theorem]{Proposition}
\theoremstyle{definition}
\newtheorem{definition}[theorem]{Definition}
\theoremstyle{remark}
\newtheorem{remark}[theorem]{Remark}

\title{Half-isolated zeros and zero-density estimates}

\author{James Maynard}
\address{Mathematical Institute, Woodstock Road, Oxford OX2 6GG, UK}
\email[James Maynard]{james.alexander.maynard@gmail.com}

\author{Kyle Pratt}
\address{All Souls College, University of Oxford, UK}
\email[Kyle Pratt]{kyle.pratt@all-souls.ox.ac.uk}

\begin{document}
%
%
%
%
\begin{abstract}
We introduce a new method to detect the zeros of the Riemann zeta function which is sensitive to the vertical distribution of the zeros. This allows us to prove there are few `half-isolated' zeros. By combining this with classical methods, we improve the Ingham-Huxley zero-density estimate under the assumption that the non-trivial zeros of the zeta function are restricted to lie on a finite number of fixed vertical lines. This has new consequences for primes in short intervals under the same assumption.
\end{abstract}
%
%
%
%
\maketitle
%
%
%
%
%
%
%
%
\section{Introduction}
%
%
%
%
The Riemann Hypothesis is the claim that all of the non-trivial zeros of the Riemann zeta function $\zeta(s)$ have real part equal to $\frac{1}{2}$, and this would have important consequences for the distribution of prime numbers. One consequence would be an asymptotic formula for primes in short intervals; for any function $f(x)\rightarrow\infty$, we could obtain an asymptotic formula for the number of primes in the interval $[x,x+f(x)\sqrt{x}(\log x)]$ under the assumption of the Riemann Hypothesis.

For many of these applications to the distribution of primes, one can still obtain results of a similar strength even if there are some zeros with real part larger than $\frac{1}{2}$, provided these exceptions are sufficiently rare. If we define the counting function
\begin{align*}
N(\sigma,T) := |\{ \rho = \beta +i\gamma : \zeta(\rho) = 0,\, \beta \geq \sigma,\,0\leq \gamma \leq T\}|,
\end{align*}
then a bound of the form $N(\sigma,T)\ll T^{A(1-\sigma)+o(1)}$ for a fixed constant $A$ would imply an asymptotic formula for primes in the interval $[x,x+x^{1-1/A+o(1)}]$. In particular, the `Density Hypothesis' which claims $N(\sigma,T)\ll T^{2(1-\sigma)+o(1)}$ would imply an asymptotic formula for primes in the interval $[x,x+x^{1/2+o(1)}]$, which is almost as strong as the result obtained under the Riemann Hypothesis.

Although the Density Hypothesis remains unproven, weaker results in this direction can be obtained unconditionally. After some early work by Carlson \cite{Car1921}, Ingham \cite{Ing1940} proved the estimate
\begin{equation}
N(\sigma,T) \ll T^{3(1-\sigma)/(2-\sigma)+o(1)}.
\label{eq:InghamEstimate}
\end{equation}
Huxley \cite{Hux1972} built on work of Montgomery \cite{Mon1971} and Hal\'asz \cite{Hal1968}, and proved the estimate
\begin{align}
N(\sigma,T) \ll T^{3(1-\sigma)/(3\sigma-1)+o(1)},
\label{eq:HuxleyEstimate}
\end{align}
which is superior to Ingham's estimate for $\sigma> 3/4$. Taken together, these bounds imply 
\begin{align}
N(\sigma,T) &\ll T^{\frac{12}{5}(1-\sigma)+o(1)},\label{eq:ZeroDensity}
\end{align}
and so as  consequence one has an asymptotic formula for the number of primes in an interval $[x,x+x^{7/12+o(1)}]$. This remains essentially the best asymptotic result on primes in short intervals to date (Heath-Brown \cite{HB1988} refined the $o(1)$ term via sieve methods). We note that any improvement to the exponent in \eqref{eq:ZeroDensity} in the neighbourhood of $\sigma=3/4$ would yield a stronger result. Various authors have improved estimates when $\sigma>3/4$, but improving zero density estimates for $\sigma\le 3/4$ remains a large open problem in analytic number theory which has seen essentially no progress in the 80 years since Ingham's estimate (beyond refinements to the $o(1)$ term).

Even if one cannot improve upon \eqref{eq:ZeroDensity}, one might hope to obtain an improved estimate for primes in short intervals by considering the vertical distribution of zeros close to the line $\text{Re}(s)=3/4$. If one could show suitable cancellation in sums similar to
\[
\sum_{\substack{3/4-\delta\le\text{Re}(\rho)\le 3/4+\delta\\ 0\le \text{Im}(\rho)\le x^{5/12+\delta} }}x^{\rho-1}
\]
then one could show that we obtain an asymptotic formula for primes in the interval $[x,x^{7/12-\delta+o(1)}]$. However, there are putative bad configurations of zeros where the sum would exhibit essentially no cancellation. For example, for $T= x^{5/12+\delta}$ one would need to rule out the possibility that there are $T^{3/5}$ zeros in the set 
\[
\Bigl\{\frac{3}{4}+i\frac{2\pi n}{\log{x}}:\,n\in \mathbb{Z}\cap[0,T]\Bigr\}.
\]
Thus vertical arithmetic progressions of zeros present obstacles to progress on results about primes.  

Our first main result allows us to improve on these zero density results, if we assume some rigidity in the real parts of the zeros of $\zeta(s)$. We therefore introduce the following hypothesis.
%
%
%
%
\begin{hypothesisF}
The non-trivial zeros of the zeta function lie on a finite number of fixed vertical lines.
\end{hypothesisF}
%
%
%
%
We note that even under the assumption of Hypothesis $\mathcal{F}$, we have not ruled out configurations of zeros such as the vertical arithmetic progressions mentioned earlier.
%
%
%
%
\begin{theorem}[Density Hypothesis under Hypothesis $\mathcal{F}$]\label{thm:ZeroDensity}
Assume Hypothesis $\mathcal{F}$. Then we have 
\begin{align*}
N(\sigma,T) &\ll T^{2(1-\sigma)} \exp((\log \log T)^{O(1)}).
\end{align*}
\end{theorem}
%
%
%
%
An immediate consequence of this result is a conditional improvement to primes in short intervals. It is well-known that the Density Hypothesis yields a result on primes in short intervals nearly as strong as that coming from the Riemann Hypothesis.
%
%
%
%
\begin{corollary}[Primes in short intervals under Hypothesis $\mathcal{F}$]\label{cor:ShortIntervals}
Assume Hypothesis $\mathcal{F}$ and let $\epsilon>0$. Then
\begin{align*}
\#\{p\in [x,x+x^{1/2+\epsilon}]\}=(1+o_\epsilon(1))\frac{x^{1/2+\epsilon}}{\log{x}}.
\end{align*}
\end{corollary}
%
%
%
%
Another consequence of Theorem \ref{thm:ZeroDensity} is a conditional improvement on primes in \emph{almost all} short intervals.
%
%
%
%
\begin{corollary}[Primes in almost all short intervals under Hypothesis $\mathcal{F}$]\label{cor:AlmostAllShortIntervals}
Assume Hypothesis $\mathcal{F}$. Let $C>0$ be a sufficiently large constant, and let $y = \exp((\log \log X)^C)$. Then
\begin{align*}
\#\{p\in [x,x+y]\}=(1+o(1))\frac{y}{\log{x}}
\end{align*}
for all but $O(X\exp(-(\log\log X)^{-2}))$ integers $x \in [X,2X]$.
\end{corollary}
%
%
%
%
One should be able to relax Hypothesis $\mathcal{F}$ considerably and still obtain similar results via our method. For example, we could allow for zeros with ordinate in $[T,2T]$, say, to lie on $\ll (\log T)/(\log \log T)^2$ well-spaced vertical lines. We have chosen to just present the simplest case in the interest of clarity. However, at our current state of knowledge it appears that some assumption of rigidity of real parts in necessary - see Section \ref{sec:Bows} for a discussion of this.

Underlying Theorem \ref{thm:ZeroDensity} is a new unconditional way of detecting certain `half-isolated' zeros of $\zeta(s)$. The key point is that this method is sensitive to the vertical distribution of zeros.  The precise definition of a half-isolated zero is somewhat technical (see Definition \ref{def:HalfIsolated} below), but roughly speaking $\rho_0=\beta_0+i\gamma_0$ is half-isolated if all nearby zeros have real part very close to $\beta_0$ and imaginary part at least $\gamma_0$ (so these zeros lie vertically `above' $\rho_0$). We prove, unconditionally, that half-isolated zeros have very short zero-detecting polynomials.
%
%
%
%
\begin{theorem}[Half-isolated zeros have short zero-detecting polynomials]\label{thm:HalfIsolated}
There exists an absolute constant $C\geq 1$ and a fixed non-negative smooth function $w_0$ supported in $[1/2,2]$ such that the following holds.

If $\rho_0=\beta_0+i\gamma_0$ is a half-isolated zero (see Definition \ref{def:HalfIsolated} for the precise statement) with $\gamma_0 \in [T,2T]$, then there exists a real $Y \in [T^{(\log\log{T})^3/\log{T}},T^{5/\log\log{T}}]$ such that
\begin{align*}
\Bigl|\sum_{n} \frac{\Lambda(n)}{n^{\rho_0}} w_0(n/Y) \Bigr| \geq (\log T)^{-C}.
\end{align*}
\end{theorem}
%
%
%
%
Since short Dirichlet polynomials rarely take large values, one immediate consequence of the existence of short zero-detecting polynomials for half-isolated zeros is that there are few half-isolated zeros.
%
%
%
%
\begin{corollary}[Half-isolated zeros satisfy Density Hypothesis]\label{cor:half isolated zeroes satisfy DH}
The number of half-isolated zeros $\rho_0 = \beta+i\gamma$ such that $\beta \geq \sigma$ and $\gamma \in [T,2T]$ is at most $T^{2(1-\sigma)+o(1)}$.
\end{corollary}
%
%
%
%
Similar results on efficiently detecting certain zeros of $\zeta(s)$ have been obtained in earlier work of Heath-Brown (unpublished), Balasubramanian and Ramachandra \cite{BS1982}, and Conrey and Iwaniec \cite{CI2014}. The key distinguishing feature of our work is that we can detect zeros even when there are many other zeros nearby, and this allows us to address possible obstructions such as vertical arithmetic progressions of zeros.

We hope that the ideas underlying this estimate might evenually be extended to give an improvement to the Huxley-Ingham bound \eqref{eq:ZeroDensity}. Unfortunately we have been unable, thus far, to exploit half-isolated zeros to unconditionally provide improvements to bounds for $N(\sigma,T)$; see Section \ref{sec:Bows} for a discussion of what we perceive to be the key obstacles to obtaining an unconditional version of Theorem \ref{thm:ZeroDensity}.

Lastly, we mention a more speculative result. One may hypothesize about the vertical distribution of the zeros when Hypothesis $\mathcal{F}$ occurs. Since the zeros are constrained to vertical lines, and any vertical line with real part larger than 1/2 must have relatively few zeros (by the Ingham-Huxley bounds), we see that the zeros must occur in separate `clusters' of nearby zeros (where a zero of height $T$ is `nearby' another if it is within $(\log{T})^3$, say, of the other zero.) One might naively expect that if the zeros were somehow randomly distributed on the vertical lines, then any cluster of zeros would have to be quite short. If this is indeed the case, then we obtain stronger information on moments of the Riemann zeta function.
%
%
%
%
\begin{theorem}\label{thm:FifthMoment}
Assume Hypothesis $\mathcal{F}$, and assume that every cluster of zeros (see Definition \ref{dfntn:Cluster} for a precise statement) has size $\leq T^{\varepsilon_0}$ for some $\varepsilon_0 \geq 0$. Then
\begin{align*}
\int_0^T |\zeta(\tfrac{1}{2}+it)|^5 dt \ll_{\varepsilon_0} T^{1+O(\varepsilon_0)+o(1)}.
\end{align*}
\end{theorem}
%
%
%
%
We leave the proof of Theorem \ref{thm:FifthMoment} to an appendix since it has a slightly different flavour from the rest of our paper.
%
%
%
%
\section{Outline}
%
%
%
%
\subsection{Zero detecting polynomials}
%
%
%
%
The methods of Ingham and Huxley are based on showing that if $\rho$ is a zero of $\zeta(s)$, then there must be a Dirichlet polynomial which takes a large value at $\rho$; one then shows that these Dirichlet polynomials rarely take large values. The proof of the existence of the zero-detecting polynomial does not give much control over the length of the Dirichlet polynomial produced, but the bounds for the frequency of large values are quite sensitive to the length of the Dirichlet polynomial. The method works well if the Dirichlet polynomial has a convenient length, but the current estimates are limited by large value estimates for Dirichlet polynomials of inconvenient length (for $\sigma=3/4$ the critical case is Dirichlet polynomials of length $T^{2/5}$). If the Dirichlet polynomial has length $T^\epsilon$, for example, then the number of large values on the line $\text{Re}(s)=\sigma$ is $O(T^{2(1-\sigma)+\epsilon+o(1)})$, giving Density Hypothesis-strength bounds when $\epsilon$ is small.

There is another, less well-known, circle of ideas for detecting zeros of the zeta function, going back to the above-mentioned work of Heath-Brown, Balasubramanian and Ramachandra, and Conrey and Iwaniec. This is our starting point.

Choose a smooth function $w$ with compact support, and for a parameter $Y>0$ we consider the Dirichlet polynomial
\begin{align*}
S_Y(s):=\sum_n \frac{\Lambda(n)}{n^{s}}w(n/Y).
\end{align*}
We use Mellin inversion to rewrite the sum as an integral involving $-\zeta'/\zeta$, and then shift the line of integration and pick up contributions from the zeros of $\zeta$. Hence
\begin{align*}
S_Y(s)=\sum_n \frac{\Lambda(n)}{n^{s}}w(n/Y) \approx \sum_{\rho} Y^{\rho - s} W(\rho - s),
\end{align*}
where $W$ is the Mellin transform of $w$ and $\rho$ ranges over non-trivial zeros of $\zeta$. The rapid decay of $W$ reduces the sum over zeros to zeros $\rho$ which are close to $s$. If $s=\rho_0$ is a zero of $\zeta$, then there is a large contribution $W(0)$ from $\rho=\rho_0$, and so we find
\begin{align}
S_Y(\rho_0)\approx W(0)+\sum_{0<|\rho-\rho_0| \text{ small}} Y^{\rho - \rho_0} W(\rho - \rho_0).\label{eq:ZeroSum}
\end{align}
If there are no other nearby zeros, then  the sum above is empty and $S_Y(s)$ takes a large value at $\rho_0$. Therefore $S_Y(s)$ is a zero detecting polynomial with controlled length $Y$, and by choosing $Y=T^\epsilon$ we are in a favourable situation for our large-values estimates. This gives Heath-Brown's result that there are few `isolated' zeros. The work of Balasubramanian-Ramachandra and Conrey-Iwaniec used Tur\'an power-sum estimates to show that $S_Y(\rho_0)$ is large for some suitable $Y$ provided there are few zeros in the neighbourhood of $\rho_0$. However, if a zero $\rho_0=\beta_0 + i\gamma_0$ has $\asymp \log T$ zeros nearby (assuming $\gamma_0 \approx T$), then this is too many zeros for the power sum methods to handle.
%
%
%
%
\subsection{Improved power-sum estimates}
%
%
%
%
Following the above strategy, by considering the distribution of zeros close to $\rho_0$ one might hope to be able to choose $Y$ to avoid the `inconvenient' lengths which arise in the Ingham-Huxley methods and still show that $S_Y(\rho_0)$ is large. Unfortunately this is not possible in general; if the zeros in the neighbourhood of $\rho_0$ are approximately in a vertical arithmetic progression with common difference $2\pi c i/\log{T}$ then $S_Y(\rho_0)$ is small for all choices of smooth $w$ and $Y<T^{1/c-\epsilon}$. Therefore vertical arithmetic progressions again indicate a barrier to this approach, and this potential barrier is not ruled out by Hypothesis $\mathcal{F}$.

One might speculate that the \textit{only} obstruction to $S_Y(\rho_0)$ being uniformly small in this way is caused by the local distribution of zeros being a pertubation of a small number of vertical arithmetic progressions. This would mean that the only obstructions to the approach are quite rigid, and in particular half-isolated zeros would be detectable since locally there would not be any arithmetic progressions of zeros extending below the half-isolated zero.

The key technical ingredient in the deduction of Theorem \ref{thm:HalfIsolated} is an improved power sum result which can be adapted to detect half-isolated zeros. This then does allow us to deduce that there is a choice of small $Y$ where $S_Y(\rho_0)$ is large when $\rho_0$ is half-isolated. There are several technical assumptions in the full statement of the result, but the following special case is indicative.
%
%
%
%
\begin{lemma}[Simple case of power sum inequality]
Let $\theta_1\leq \cdots \leq \theta_R$ be real numbers. There exists an absolute constant $B_0 \geq 1$ such that the following is true. For any $A >0$, there exists a real $t\in [A,2A]$ such that
\begin{align*}
\Bigl|\sum_{r=1}^R \exp(it\theta_r) \Bigr| \geq (B_0 R)^{-99}.
\end{align*}
\end{lemma}
%
%
%
%
Traditional power-sum methods would yield a lower bound which is exponentially small in $R$, but versions of this result have appeared before. See Section \ref{sec:HalfIsolated} for a full statement of the result and further discussion.
%
%
%
%
\subsection{Local-global structure using Hypothesis \texorpdfstring{$\mathcal{F}$}{F}}
%
%
%
%
A key point of Hypothesis $\mathcal{F}$ is that if the zeros lie on finitely many vertical lines, then zeros come in vertical `clusters', with a half-isolated zero at the bottom of each cluster (see Definition \ref{dfntn:Cluster} for a precise definition of \emph{cluster}, but imagine two zeros $\rho,\rho'$ of height $T$ are in the same cluster if there is a sequence of zeros $\rho_1,\dots,\rho_j$ on the same vertical line with $\rho_1=\rho$, $\rho_j=\rho'$ and $|\rho_i-\rho_{i+1}|\le (\log{T})^3$ for all $1\le i<j$.)

We exploit the clustering phenomenon by averaging over zeros within a single cluster, as well as averaging over clusters themselves. To do this successfully requires a combination of the flexible zero-detecting polynomials for half-isolated zeros, as well as the more classical zero-detecting techniques.
%
%
%
%
\section*{Acknowledgements}
We thank Roger Heath-Brown for comments on an earlier version of this paper, and for his generosity in allowing us to incorporate his ideas on an improved arrangement for bounding the relevant Dirichlet polynomials in Section \ref{sec:ClusteringBound}. This led to a stronger bound for $N(\sigma,T)$ and a cleaner proof compared with our earlier version.

 JM is supported by a Royal Society Wolfson Merit Award, and this project has received funding from the European Research Council (ERC) under the European Union’s Horizon 2020 research and innovation programme (grant agreement No 851318).
%
%
%
%
\section{Notation}

It is always the case that $T$ denotes a large real number.

We utilize the traditional Landau and Vinogradov notations $O(\cdot), o(\cdot), \ll$, and $\gg$. We say $f = O(g)$ if there is a constant with $|f| \leq Cg$, and $f\ll g$ or $g \gg f$ mean $f = O(g)$. If the implied constant in a $O(\cdot)$ or $\ll$ depends on some other quantity we sometimes indicate this with a subscript, such as $f \ll_\varepsilon g$. We write $o(1)$ for a quantity that tends to zero as $T$ goes to infinity.

We write $n \sim N$ for the condition $N < n \leq 2N$.

For a function $f: \mathbb{R}_{\geq 0} \rightarrow \mathbb{C}$ with suitable convergence properties we define the Mellin transform
\begin{align*}
F(s) = \int_0^\infty f(x) x^{s-1} dx.
\end{align*}
In general, we follow the convention of using a capital letter to denote the Mellin transform of a function.

Throughout the paper $w_0(x)$ denotes a fixed, non-negative smooth function supported in $[\frac{1}{2},2]$ with certain bounds on its derivatives. We construct such a $w_0$ in Lemma \ref{lem:Partition}.

For a real number $c$ we write
\begin{align*}
\frac{1}{2\pi i}\int_{(c)} f(z) dz = \frac{1}{2\pi i}\int_{c-i\infty}^{c+i\infty} f(z) dz.
\end{align*}
We write $\text{Re}(z)$ and $\text{Im}(z)$ for the real and imaginary parts, respectively, of $z \in \mathbb{C}$, so that $z = \text{Re}(z) + i\text{Im}(z)$. If $z = re^{i\theta}$ with $r\geq 0$ and $\theta \in (-\pi,\pi]$ we write $\text{arg}(z) = \theta$.

As usual, we write $e(x) = e^{2\pi i x}$. We write $P^+(n)$ and $P^-(n)$ for the greatest and least prime factor of a positive integer $n$, where $P^+(1) = P^-(1) = 1$. We write $\tau(n)$ for the usual divisor function, and $\mu(n)$ for the M\"obius function.

A set $\mathcal{S}$ of real numbers if $G$-separated if $|\alpha - \beta| \geq G$ for all distinct $\alpha,\beta \in \mathcal{S}$. The notation $|\mathcal{S}|$ denotes the cardinality of a finite set.
%
%
%
%
\section{Half-isolated zeros}\label{sec:HalfIsolated}
%
%
%
%
In this section we introduce half-isolated zeros and then establish Theorem \ref{thm:HalfIsolated} in the more quantitative form of Proposition \ref{prp:HalfIsolated}. This section is essentially self-contained and can be read separately from the rest of the paper. We begin by giving a precise definition of \emph{half-isolated zero}.
%
%
%
%
\begin{definition}\label{def:YHalfIsolated}
Given a real $Y>1$, we call a zero $\rho_0=\beta_0+i\gamma_0$ of $\zeta(s)$ \textit{`$Y$-half-isolated'} if every zero $\rho'=\beta'+i\gamma'$ of $\zeta(s)$ with $|\rho'-\rho_0|\le (\log{|\gamma_0|})^2$ satisfies at least one of the following:
\begin{enumerate}
\item (Similar real part and larger imaginary part)
\[
|\beta'-\beta_0|\le \frac{1}{10\log{Y}}\qquad \text{and} \qquad \gamma'\ge \gamma_0.
\]
\item (Smaller real part)
\[
\beta'\le \beta_0-\frac{(\log\log{|\gamma_0|})^2}{\log{Y}}.
\]
\end{enumerate}
\end{definition}
%
%
%
%
\begin{definition}\label{def:HalfIsolated}
We call a zero $\rho_0=\beta_0+i\gamma_0$ of $\zeta(s)$ \textit{`half-isolated'} if it is $Y$-half-isolated for some $Y$ satisfying
\[
(\log\log{|\gamma_0|})^3\le \log{Y}\le \frac{\log{|\gamma_0|}}{\log\log|\gamma_0|}.
\]
\end{definition}
%
%
%
%
The reader should bear in mind the situation when $\gamma_0\in[ T,2T]$ and $Y\approx T^{1/\log\log{T}}$, in which case any zero $\beta+i\gamma$ within $(\log{T})^2$ of $\rho_0$ must satisfy either $\beta\le \beta_0-(\log\log{T})^3/\log{T}$ or $|\beta-\beta_0|\leq (\log\log{T})/(10\log{T})$ and $\gamma\ge \gamma_0$.
%
%
%
%
\begin{center}
\begin{figure}[h!!!]
\begin{tikzpicture}[every node/.style={scale=0.8}]
\fill[fill=gray!30] (0,0) rectangle (10,10);
\fill[fill=white] (5,5) circle (4.5);
\fill[fill=gray!30] (0,0) rectangle (3,10);
\fill[fill=gray!30] (4.5,10) rectangle (5.5,5);
\fill[fill=black] (5,5) circle (0.1);
\draw[<->] (3,5)--(4.9,5);
\draw[<->] (5.07,4.93)--(8.2,1.8);
\draw[<->] (4.5,7)--(5.5,7);
\node [] at (5,4.6) {$\rho_0$};
\node [align=center] at (5,2) {No nearby zeros\\ below $\rho_0$};
\node [align=center] at (7.5,6) {No nearby zeros\\ to the right of $\rho_0$};
\node [align=center] at (1.5,5) {Possible zeros\\ to the left of $\rho_0$};
\node [align=center] at (5,8.5) {Possible zeros\\ vertically above $\rho_0$};
\node [align=center] at (8.5,0.5) {Possible zeros\\ far away from $\rho_0$};
\node [] at (7.2,3.5) {$(\log{T})^2$};
\node [] at (4,4.6) {$\frac{(\log\log{T})^3}{\log{T}}$};
\node [] at (5,6.6) {$\frac{\log\log{T}}{10\log{T}}$};
\end{tikzpicture}
\caption{Possible locations of zeros near a half-isolated zero}
\end{figure}
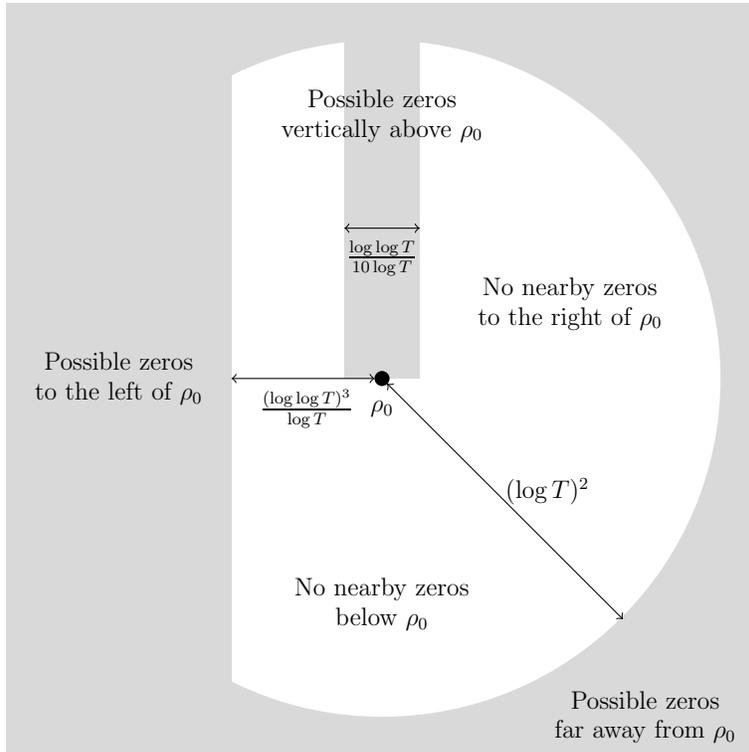
\end{center}
%
%
%
%
We use $\mathbb{H}$ to denote the upper half-plane, and $\overline{\mathbb{H}}$ the completed upper half-plane:
\begin{align*}
\mathbb{H}&:=\{x+iy:\,y> 0\},\qquad \overline{\mathbb{H}}:=\mathbb{H}\cup\mathbb{R}=\{x+iy:\,y\ge 0\}.
\end{align*}
%
%
%
%
\begin{lemma}[Subharmonic inequality for subexponential analytic functions on $\overline{\mathbb{H}}$]\label{lmm:Harmonic}
Let $f:\overline{\mathbb{H}}\rightarrow \mathbb{C}$ be an analytic function on $\mathbb{H}$ which is continuous on $\overline{\mathbb{H}}$ and such that for all $z\in\overline{\mathbb{H}}$ we have $|f(z)|\le \exp(O(|z|^{1/2}))$. Then for every $z=x+iy\in\mathbb{H}$ with $f(z)\ne 0$ we have
\[
\log|f(z)|\le \frac{y}{\pi}\int_{-\infty}^{\infty}\frac{\log|f(t)|}{(t-x)^2+y^2}dt.
\]
\end{lemma}
%
%
%
%
\begin{proof}
Let $g$ be given by
\[
g(s):=f\Bigl(\frac{iy(1+s)}{1-s}+x\Bigr).
\]
By our hypotheses on $f$ we see $g$ is analytic on $|s|<1$, continuous on $\{s:|s|\le 1,s\ne 1\}$ and $g(0)\ne 0$. Thus by Jensen's formula, for any $r \in (0,1)$ we have
\[
\log|g(0)|\le \frac{1}{2\pi}\int_0^{2\pi} \log|g(re^{i\theta})|d\theta.
\]
Since $|f(s)|\le \exp (O(|s|^{1/2}))$ we have that 
\[
\log|g(s)|\ll_{x,y} \frac{1}{|1-s|^{1/2}}.
\]
Therefore by the dominated convergence theorem we deduce that
\[
\lim_{r\rightarrow 1^-}\int_0^{2\pi} \log|g(re^{i\theta})|d\theta=\int_0^{2\pi} \log|g(e^{i\theta})|d\theta,
\] 
with this integral converging absolutely. Thus we find
\begin{align*}
\log|f(x+iy)|=\log|g(0)|&\le \frac{1}{2\pi}\int_0^{2\pi}\log|g(e^{i\theta})|d\theta=\frac{y}{\pi}\int_{-\infty}^{\infty} \frac{\log|f(t)|}{(t-x)^2+y^2}dt,
\end{align*}
where in the last equality we made the change of variables 
\begin{align*}
t&=iy(1+e^{i\theta})/(1-e^{i\theta})+x=x-y\cot(\theta/2). \qedhere
\end{align*}
\end{proof}
%
%
%
%
\begin{lemma}[Refined power sum estimate]\label{lmm:Turan}
Let $A,B>0$ with $B$ sufficiently large, and $z_1,\dots,z_R$ and $c_1,\dots,c_R$ be complex numbers, such that
\begin{enumerate}
\item $z_1=0$ and $c_1=1$
\item $\textup{Im}(z_r)\ge 0$ for all $1\le r\le R$
\item $|\textup{Re}(z_r)|\le \frac{1}{10A}$ for all $1\le r\le R$.
\item If $|\textup{Im}(z_r)|\le (\log{B})^2/A$, then $|\arg(c_r)|\le 1/10$.
\item $B\ge\sum_{r=1}^R |c_r|$.
\end{enumerate}
Then there is a real $t\in[A,2A]$ such that
\[
\Bigl|\sum_{r=1}^R c_r \exp(t z_r)\Bigr|\ge B^{-99}.
\]
\end{lemma}
%
%
%
%
\begin{remark}
The reader should have in mind the case when all the $z_r=i\theta_r$ are purely imaginary, when $c_r=1$ for all $r$ and $B=R>A$. In this case Lemma \ref{lmm:Turan} gives
\[
\sup_{t\in [A,2A]}\Bigl|\sum_{r=1}^{R}\exp(i t \theta_r)\Bigr|\gg R^{-99}.
\]
Standard Tur\'an power sum methods (see, for example \cite[Chapter 5, Theorem 1]{Mon1994}) would show that there is an integer $t\in[A,A+R]$ such that the sum is at least $\exp(- c R )$ for some constant $c>0$. In this case we obtain an exponential improvement in the lower bound and a tighter bound on $t$ by exploiting positivity of the coefficients and the fact that we consider $t$ real. Other variants of Tur\'an's method do exploit positivity of the coefficients (see \cite[Chapter 5, \S  6]{Mon1994}), but we are not aware of any results which allow us to restrict the range of $t$ to a short interval, which is vital for our applications.
\end{remark}
%
%
%
%
\begin{remark}\label{rmk:Turan}
Several of the conditions in Lemma \ref{lmm:Turan} which appear artificial at first sight are in fact necessary in some form.
\begin{enumerate}
\item The relaxation to real values of $t$ is important. If $z_r=2\pi i(r-1)/R$ and $c_r=1$ for all $r\le R$, then the sum vanishes for all integers $t\in(0,R)$.  
\item The positivity condition that $|\arg(c_r)|\le 1/10$ if $|z_r|\le (\log{B})^2/A$ is important. If we allow $c_r\in\{-1,1\}$ then by choosing $z_r=i\theta_r$ and $c_r$ such that 
\[
\sum_{r=1}^{2^k}c_r \exp(it\theta_r)=\Bigl(1-e\Bigl(\frac{t}{20A^2}\Bigr)\Bigr)^k,
\]
 we see that the sum has size $\leq (2^k)^{-\log A}$ for all $t\in [A,2A]$.
\item The discreteness condition that $c_1=1$ and $\text{Im}(z_r)\ge \text{Im}(z_1)$ is important. If we choose a smooth non-negative function $f$ supported on $[0,1]$, $c_r=f(r/R)$ and $z_r=\pi i r/(2A)$, then we see that by Poisson summation
\[
\sum_{r=1}^R c_r \exp(t z_r)=\sum_{r\in\mathbb{Z}}f\Bigl(\frac{r}{R}\Bigr)e\Bigl(\frac{rt}{4A}\Bigr)=R\sum_{\ell\in\mathbb{Z}}\hat{f}\Bigl(R\Bigl(\ell-\frac{t}{4A}\Bigr)\Bigr).
\]
Since $-t/(4A)$ is always at least 1/4 away from any integer for $t\in [A,2A]$, and $\hat{f}(\xi)$ has rapid decay since $f$ is smooth, we see that the sum is $O_j(R^{-j})$ for every $j\in \mathbb{Z}_{>0}$.
\item There are choices of $c_1,\ldots,c_R$ and $z_1,\dots,z_R$ satisfying the hypotheses of Lemma \ref{lmm:Turan} such that the sum is genuinely polynomially small for all $t\in[A,2A]$. Consider first
\[
e^{i\frac{4\pi t}{3A}}+3e^{i\frac{2\pi t}{3A}}+4+3e^{-i\frac{2\pi t}{3A}}+e^{-i\frac{4\pi t}{3A}}=|e^{i\frac{\pi t}{3A}}+e^{-i\frac{\pi t}{3A}}|^4-|e^{i\frac{\pi t}{3A}}+e^{-i\frac{\pi t}{3A}}|^2.
\]
For $t\in [A,2A]$ this is of the form $u^4-u^2$ for some real $|u|\le 1$, and hence takes values in $[-1/4,0]$. Thus if we take $R=12^{k}$ and choose $c_r=1$, $z_r=i\theta_r$ such that
\[
\sum_{r=1}^R \exp(i t\theta_r)=\Bigl(e^{i\frac{4\pi t}{3A}}+3e^{i\frac{2\pi t}{3A}}+4+3e^{-i\frac{2\pi t}{3A}}+e^{-i\frac{4\pi t}{3A}}\Bigr)^k,
\]
we see that the sum is $\leq 4^{-k}=R^{-\log{4}/\log{12}}=R^{-0.55\dots}$ for all $t\in [A,2A]$. This example was shown to us by Jean Bourgain.
\end{enumerate}
\end{remark}
%
%
%
%
\begin{remark}
A version of Lemma \ref{lmm:Turan} was independently shown by Jean Bourgain (private communication to the first author), with essentially the same proof. Will Sawin \cite{SawMO} also independently proved a version of Lemma \ref{lmm:Turan}, again with fundamentally the same proof. It seems plausible a version of this lemma has existed in the literature for some time; Terence Tao pointed out to us that this can be viewed as a consequence of a quantified version of the F. and M. Riesz Theorem. 
\end{remark}
%
%
%
%
\begin{proof}[Proof of Lemma \ref{lmm:Turan}]
Let $f:\mathbb{C}\rightarrow\mathbb{C}$ be defined by
\[
f(s):=\sum_{r=1}^R c_r e^{s z_r}.
\]
Assume for a contradiction that $|f(t)|\le B^{-99}$ for all $t\in [A,2A]$.

Let $z_r=x_r+iy_r$. We wish to apply Lemma \ref{lmm:Harmonic}, but if $x_r\ne 0$ for some $r$ in the sum then $f$ can have exponential growth, so we first take a Taylor series approximation. Since $|x_r|\le 1/(10A)$, we see that uniformly for $|s|\le 2A$
\[
e^{s x_r}=\sum_{j=0}^J\frac{s^j x_r^j}{j!}+O(B^{-2\log\log{B}}),
\]
where 
\[
J:=\lfloor 3\log{B}\rfloor.
\]
We therefore define $\widetilde{f}:\mathbb{C}\rightarrow\mathbb{C}$ by
\[
\widetilde{f}(s):=\sum_{r=1}^R c_r e^{i s y_r}\sum_{j=0}^J\frac{s^j x_r^j}{j!}.
\]
Since $\sum_{r}|c_r|\le B$, we have that for $s \in \mathbb{H}$ with $|s|\le 2A$
\begin{equation}
\widetilde{f}(s)=f(s)+O(B^{-\log\log{B}}).
\label{eq:TildeFApprox}
\end{equation}
Since we sum over a finite set, $\widetilde{f}$ is analytic on $\mathbb{C}$. Since $y_r\ge 0$ for all $r$, $\text{Re}(isy_r)\le 0$ for $s\in \mathbb{H}$, so
\begin{equation}
\widetilde{f}(s)\ll_{A,B} 1+ |s|^J\ll_{A,B} \exp(|s|^{1/2})
\label{eq:TildeFGrowth}
\end{equation}
 on $\mathbb{H}$. In particular, we may apply Lemma \ref{lmm:Harmonic} to $\widetilde{f}$ and find
\begin{equation}
\log|\widetilde{f}(i A)|\le \frac{A}{\pi}\int_{-\infty}^{\infty}\frac{\log{|\widetilde{f}(t)|}}{t^2+A^2}dt.
\label{eq:TildeFHarmonicBound}
\end{equation}
We wish to obtain a contradiction by obtaining a lower bound for $\log|\widetilde{f}(iA)|$ and an upper bound for the integral. First we note that since $B$ is large
\begin{align}
\widetilde{f}(i A)&=f(i A)+O(B^{-200})=\sum_{r=1}^R c_r \exp(-A y_r)\exp(i A x_r)+O(B^{-200}).
\label{eq:iAexpansion}
\end{align}
 We see that if $y_r>(\log{B})^2/(2A)$, then 
\[
c_r \exp(-A y_r)\exp(i A x_r)\ll B\exp(-(\log{B})^2/2)\le B^{-201}.
\]
Therefore the contribution of $r$ with $y_r>(\log{B})^2/(2A)$ to \eqref{eq:iAexpansion} is negligible. By assumption of the lemma, if $y_r\le (\log{B})^2/(2A)$ then $|\arg(c_r)|\le 1/10$. Since $|x_r|\le 1/(10A)$, we also have that
\[
|\arg(\exp(i A x_r))|\le \frac{1}{10}.
\]
In particular, for all $r$ such that $y_r\le (\log{B})^2/(2A)$ we have
\[
\text{Re}( c_r \exp(-Ay_r)\exp(i A x_r)  )\ge 0.
\]
Therefore, by taking real parts of \eqref{eq:iAexpansion} and separating the term $z_1=0$ (which contributes $1$) and dropping other terms with $y_r\le (\log{B})^2/(2A)$ for a lower bound, we find
\begin{equation}
|\widetilde{f}(i A)|\ge \text{Re}(\widetilde{f}(i A))\ge 1-O(B^{-200}).
\label{eq:TildeFLowerBound}
\end{equation}
We now consider the integral on the right hand side of \eqref{eq:TildeFHarmonicBound}. By assumption, $\log|f(t)|\le -99\log{B}$ for $t\in[A,2A]$. Therefore by \eqref{eq:TildeFApprox} $\log|\widetilde{f}(t)|\le -99\log{B}+O(1)$ on $[A,2A]$, and so
\begin{equation}
\frac{A}{\pi}\int_{A}^{2A}\frac{\log{|\widetilde{f}(t)|}}{t^2+A^2}dt\le \frac{-99\log{B}}{5\pi}+O(1).\label{eq:TildeFInt1}
\end{equation}
Since $|x_r|\le 1/(10A)$ for all $r$, on the real axis we have the bound
\begin{equation}
\widetilde{f}(x)\ll B+B\sup_{1\le j\le J}\Bigl(\frac{|x|}{j A}\Bigr)^j\ll 
\begin{cases}
B\exp\left(\displaystyle\frac{|x|}{A}\right),& |x|\le A(\log{B})^2,\\
\left(\displaystyle\frac{|x|}{A}\right)^{3\log{B}},&|x|>A(\log{B})^2.
\end{cases}
\label{eq:TildeFUpperBound}
\end{equation}
These bounds give
\begin{align}
\frac{A}{\pi}\int_{\substack{|t|\le A(\log{B})^2\\ t\notin [A,2A]}}\frac{\log{|\widetilde{f}(t)|}}{t^2+A^2}dt&\le \log{B}+O(\log\log{B}),\label{eq:TildeFInt2}\\
\frac{A}{\pi}\int_{|t|>A(\log{B})^2}\frac{\log{|\widetilde{f}(t)|}}{t^2+A^2}dt&\le O(1).\label{eq:TildeFInt3}
\end{align}
Putting together \eqref{eq:TildeFHarmonicBound}, \eqref{eq:TildeFInt1}, \eqref{eq:TildeFInt2} and \eqref{eq:TildeFInt3} gives
\[
\log|\widetilde{f}(i A)|\le\frac{A}{\pi}\int_{-\infty}^{\infty}\frac{\log{|\widetilde{f}(t)|}}{t^2+A^2}dt\le -\log{B}.
\] 
But this contradicts our lower bound \eqref{eq:TildeFLowerBound} that $|\widetilde{f}(i A)|\ge 1-O(B^{-100})$. Thus there must be some $t\in [A,2A]$ such that $|f(t)|\ge B^{-99}$, giving the result.
\end{proof}
%
%
%
%
\begin{lemma}[Large values of the zero sum]\label{lmm:ZeroSum}
Let $T$ and $Y$ be sufficiently large and satisfy
\[
(\log\log{T})^{3}\le \log{Y}\le 100\log{T}.
\]
Let $W:\mathbb{C}\rightarrow\mathbb{C}$ be analytic with $W(0)=1$ and $|W(s)|,|W'(s)|\ll \min(1,|s|^{-2})$ for $\textup{Re}(s)\in[-1,1]$.

Let $\rho_0=\beta_0+i\gamma_0$ be a zero of $\zeta(s)$ with $\gamma_0\in[T,2T]$. Define
\begin{align*}
\mathcal{S}_{\approx}(\rho_0)&:=\Bigl\{\rho=\beta+i\gamma:\,\zeta(\rho)=0,\,\gamma_0\le \gamma\le \gamma_0+(\log{T})^2,\, |\beta-\beta_0|\le \frac{1}{10\log{Y}}\Bigr\},\\
\mathcal{V}&:=[Y,Y^2]\cap \Bigl\{\Bigl(1+\frac{1}{(\log{T})^{150}}\Bigr)^j:\,j\in\mathbb{Z}\Bigr\}.
\end{align*}
Then  there is a $Z\in\mathcal{V}$ such that
\[
\Bigl|\sum_{\rho\in \mathcal{S}_{\approx}(\rho_0)}W(\rho-\rho_0)Z^{(\rho-\rho_0)}\Bigr|\ge \frac{1}{(\log{T})^{100}}.
\]
\end{lemma}
%
%
%
%
\begin{proof}
Let $f:\mathbb{R}\rightarrow\mathbb{C}$ be defined by
\[
f(t):=\sum_{\rho\in \mathcal{S}_{\approx}(\rho_0)}W(\rho-\rho_0)e^{t(\rho-\rho_0)}.
\]
Let the elements of $\mathcal{S}_{\approx}(\rho_0)$ be $\rho_1,\dots,\rho_R$, with $\rho_1=\rho_0$, and let $z_r:=\rho_r-\rho_0$, $c_r:=W(\rho_r-\rho_0)$ and $A:=\log{Y}$. Then certainly $z_1=0$, $\text{Im}(z_r)\ge 0$ and $|\text{Re}(z_r)|\le 1/(10A)$ by the definition of $\mathcal{S}_{\approx}(\rho_0)$. By the decay of $W$, we see that
\[
\sum_{\rho\in \mathcal{S}_{\approx}(\rho_0)}|W(\rho-\rho_0)|\ll \sum_{\rho:\,\zeta(\rho)=0}\min\Bigl(1,\frac{1}{|\rho-\rho_0|^2}\Bigr)\ll \log{T}.
\]
Therefore we can choose $B\asymp \log{T}$ such that $\sum_{r=1}^R|c_r|\le B$. Finally, since $\log{Y}\ge (\log\log{T})^3$, we see that if $|s|\ll (\log{B})^2/A \ll 1/\log \log T$ then
\[
W(s)=1+O\Bigl(\frac{1}{\log\log{T}}\Bigr),
\]
and so for $T$ sufficiently large we have $|\arg(c_r)|\le 1/10$ whenever $|\text{Im}(z_r)|\le (\log{B})^2/A$. Therefore all the conditions of Lemma \ref{lmm:Turan} are satisfied, and so there is a $t_0\in[\log{Y},2\log{Y}]$ such that
\[
|f(t_0)|\gg \frac{1}{(\log{T})^{99}}.
\]
By definition, if $\rho\in \mathcal{S}_{\approx}(\rho_0)$ then $|\rho-\rho_0|\le (\log{T})^2$. Therefore, if $t\leq 3 \log Y$ then $|e^{t(\rho-\rho_0)}| \ll 1$ for all $\rho \in \mathcal{S}_\approx(\rho_0)$ and
\begin{align*}
|f'(t)|\ll \sum_{\rho\in \mathcal{S}_{\approx}(\rho_0)}\Bigl|W(\rho-\rho_0)(\rho-\rho_0)e^{t(\rho-\rho_0)}\Bigr|\ll (\log T)^2 \sum_{\rho : \zeta(\rho) = 0} \min \left(1, \frac{1}{|\rho-\rho_0|^2}\right) \ll (\log T)^3.
\end{align*}
We recall that $|f(t_0)|\gg (\log{T})^{-99}$. By our derivative bound we see that $|f(t)|\ge (\log{T})^{-100}$ whenever $|t-t_0|\le (\log{T})^{-110}$. Since $\log(1+(\log{T})^{-150})\le (\log{T})^{-150}$, there is a $Z\in \mathcal{V}$ such that $|\log{Z}-t_0|\le (\log{T})^{-120}$. This gives the result.
\end{proof}
%
%
%
%
\begin{proposition}[Short zero detecting polynomial for $Y$-half-isolated zeros]\label{prp:HalfIsolated}
Let $T$ be sufficiently large. Then there is a set $\mathcal{U}=\mathcal{U}(T)$ with $\mathcal{U}\subseteq [\exp((\log\log{T})^3),T^2]$ and $\#\mathcal{U}\le (\log{T})^{152}$ such that the following holds.

Let $w_0$ be the smooth function of Lemma \ref{lem:Partition}. Let $(\log \log T)^3 \leq \log Y \leq \log T$. For any $Y$-half-isolated zero $\rho_0=\beta_0+i\gamma_0$ with $\gamma_0\in[T,2T]$, there is a choice of $U\in\mathcal{U} \cap [Y,Y^2]$ such that 
\[
\Bigl|\sum_{n}\frac{\Lambda(n)}{n^{\rho_0}}w_0\Bigl(\frac{n}{U}\Bigr)\Bigr|\gg (\log{T})^{-100}.
\]
\end{proposition}
%
%
%
%
\begin{proof}
We choose the set $\mathcal{U}$ to be
\[
\mathcal{U}:=[\exp((\log\log{T})^3),T^2]\cap\Bigl\{\Bigl(1+\frac{1}{(\log{T})^{150}}\Bigr)^j:\,j\in\mathbb{Z}\Bigr\}.
\]
Then it is clear that $\#\mathcal{U}\le (\log{T})^{152}$, so it remains to show that for every $Y$-half-isolated zero $\rho_0$ there is a suitable Dirichlet polynomial in terms of some $U\in\mathcal{U}$.

 Let $\rho_0=\beta_0+i\gamma_0$ be a $Y$-half-isolated zero of $\zeta(s)$ with $\gamma_0\in [T,2T]$ and $\beta_0\ge 1/2$. Then all zeros $\rho=\beta+i\gamma$ with $|\rho-\rho_0|\le (\log{T})^2$ satisfy either $\beta\le \beta_0-(\log\log{T})^2/\log{Y}$ or $\gamma\ge \gamma_0$ and $|\beta-\beta_0|\le 1/(10\log{Y})$. Given such a $Y$, we note that the set $\mathcal{V}$ defined in Lemma \ref{lmm:ZeroSum} is a subset of $\mathcal{U}$.

 We consider the function $S:\mathbb{R}\rightarrow\mathbb{C}$, given by
\begin{align*}
S(U) := \sum_n \frac{\Lambda(n)}{n^{\rho_0}} w_0\Bigl(\frac{n}{U} \Bigr)
\end{align*}
for $U \in [Y,Y^2]$. We rewrite $w_0$ using Mellin inversion, giving
\begin{align*}
S(U) = \frac{1}{2\pi i}\int_{(1)} U^s W_0(s)\Big[ -\frac{\zeta'}{\zeta}(\rho_0 + s)\Big] ds.
\end{align*}
We shift the line of integration to $\text{Re}(s) = -2$, picking up contributions from simple poles at $s = 1-\rho_0$ and $s = \rho-\rho_0$, where $\rho$ ranges over non-trivial zeros of the zeta function. It follows that
\begin{align*}
S(U) &= U^{1-\rho_0}W_0(1-\rho_0) - \sum_{\rho}U^{\rho - \rho_0} W_0(\rho-\rho_0) + \frac{1}{2\pi i}\int_{(-2)} U^s W_0(s)\Big[ -\frac{\zeta'}{\zeta}(\rho_0 + s)\Big] ds.
\end{align*}
Since $\text{Im}(\rho_0)\geq T$ and $U\le Y^2\le T^4$, by the rapid decay of $W_0(s)$ (see Lemma \ref{lem:Partition}) we have $U^{1-\rho_0}W_0(1-\rho_0)\ll T^{-2}$. Similarly, inside the integral $W_0(s)\ll |s|^{-2}$ and 
\begin{align*}
\Bigl|\frac{\zeta'}{\zeta}(z + s) \Bigr| \ll \log (T + |s|),
\end{align*}
so (since $U\ge \log T$) the integral contributes $O(U^{-1})$ and we have
\begin{align*}
S(U) &= - \sum_{\rho}U^{\rho - \rho_0} W_0(\rho-\rho_0) + O(U^{-1}).
\end{align*}
We wish to simplify the sum over zeros $\rho$. We first note that if $|\rho-\rho_0|>(\log{T})^2/2$ then $W_0(\rho-\rho_0)\ll T^{-1/100}|\rho-\rho_0|^{-2}$ by Lemma \ref{lem:Partition}, and so these terms are negligible, giving
\[
S(U) = - \sum_{\substack{\rho\\ |\rho-\rho_0|\le (\log{T})^2/2}}U^{\rho - \rho_0} W_0(\rho-\rho_0) + O(\exp(-(\log \log T)^{10})).
\]
We now separate the contribution from zeros $\rho$ such that $\text{Re}(\rho)\le \text{Re}(\rho_0)-(\log\log{T})^2/(2\log{Y})$. Since $U\ge Y$, we see that these terms contribute
\[
\ll \exp\Bigl(-(\log\log{T})^2/2\Bigr)\sum_{\substack{\rho\\ |\rho-\rho_0|\le (\log{T})^2/2}} |W_0(\rho-\rho_0)|\ll (\log{T})^{-200}.
\]
Thus we have (noting that $U^{-1}\ll Y^{-1}\ll (\log{T})^{-200}$)
\[
S(U) = - \sum_{\substack{\rho\\ |\rho-\rho_0|\le (\log{T})^2/2\\ \text{Re}(\rho)\ge \text{Re}(\rho_0)-(\log\log{T})^2/(2\log{Y})}}U^{\rho - \rho_0} W_0(\rho-\rho_0) + O\Bigl((\log T)^{-200}\Bigr).
\]
We now use the assumption that $\rho_0$ is a $Y$-half-isolated zero to see that all the zeros $\rho=\beta+i\gamma$ appearing in the sum must satisfy
\begin{align*}
|\beta-\beta_0|\leq \frac{1}{10\log{Y}},\qquad \text{and} \qquad \gamma\ge\gamma_0.
\end{align*}
Reintroducing some negligible terms if necessary, we then see that
\[
S(U) = - (\log 2)\sum_{\substack{\rho=\beta+i\gamma\\ \gamma_0\le \gamma \le \gamma_0+(\log{T})^2\\ |\beta-\beta_0|\le 1/(10\log{Y})}}U^{\rho - \rho_0} \frac{W_0(\rho-\rho_0)}{\log 2} + O\Bigl(\frac{1}{(\log{T})^{200}}\Bigr).
\]
Lemma \ref{lem:Partition} gives $W_0(0)=\log 2$, so Lemma \ref{lmm:ZeroSum} shows that there is a choice of $U\in\mathcal{V}$ such that $S(U)\gg (\log{T})^{-100}$. Since $\mathcal{V}\subseteq\mathcal{U}$, this gives the result.
\end{proof}
%
%
%
%
Theorem \ref{thm:HalfIsolated} is a consequence of Proposition \ref{prp:HalfIsolated}, since every half-isolated zero with ordinate in $[T,2T]$ is $Y$-half-isolated for some $Y \in [\exp((\log \log T)^3), T^{2/\log \log T}]$.
%
%
%
%
\section{Zero-density estimates}\label{sec:ZeroDensity}
%
%
%
%
In this section we introduce the key results we need to prove Theorem \ref{thm:ZeroDensity}. We begin with some notation and definitions. The first definition is the notion of Type I zeros in the theory of density estimates.
%
%
%
%
\begin{definition}
Let $\rho = \beta+i\gamma$ be a non-trivial zero of the zeta function with $\gamma \in [T,2T]$. We say that $\rho$ is `detected by a Dirichlet polynomial of length $N$', if
\begin{align}\label{eq:NDetector}
|D_N(\rho)| \ge \frac{1}{3\log T},
\end{align}
where
\begin{align}
D_N(s)&:=\sum_{n\sim N}\frac{a(n)}{n^s}\exp \Bigl( -\frac{n}{T^{1/2}}\Bigr),\label{eq:DNDef}\\
a(n)& := \sum_{\substack{d \mid n \\ d\leq 2T^{1/100} }} \mu(d).
\label{eq:anDef}
\end{align}
\end{definition}
%
%
%
%
Next, we need to define a cluster of zeros, as mentioned in the introduction.
%
%
%
%
\begin{definition}\label{dfntn:Cluster}
A set $ \{\rho_1 , \cdots , \rho_r\}$ of non-trivial zeros of $\zeta(s)$ with imaginary parts in $[T,2T]$ is a \emph{cluster} if
\begin{enumerate}
\item The zeros lie on the same vertical line: There is a $\sigma\in[0,1]$ such that
\[
\text{Re}(\rho_j)=\sigma\qquad (1\le j\le r).
\]
\item The imaginary parts are ordered and increase by at most $(\log{T})^3$: We have
\[
\text{Im}(\rho_j)<\text{Im}(\rho_{j+1})<\text{Im}(\rho_j)+(\log{T})^3\qquad (1\le j<r).
\]
\end{enumerate}
\end{definition}
%
%
%
%
We remark that the zeros in a cluster are taken without multiplicity. Indeed, the multiplicity of zeros has no import in our arguments, since the bound $N(T+1)-N(T)\ll \log T$ means multiplicities only contribute harmless logarithmic factors.
%
%
%
%

Furthermore:
\begin{itemize}
\item  We say that two clusters $\mathcal{C}_1,\mathcal{C}_2$ are \emph{separated} if $|\text{Im}(t_1)-\text{Im}(t_2)| \geq (\log T)^3$ or if $\text{Re}(t_1)\ne \text{Re}(t_2)$ for all $t_1 \in \mathcal{C}_1$ and $t_2 \in \mathcal{C}_2$.
\item We say that $\mathcal{C}$ is a maximal cluster if there is no cluster $\mathcal{C}'$ strictly containing $\mathcal{C}$.
\end{itemize}
We note that a maximal cluster $\mathcal{C}$ is separated from any cluster which is not a subset of $\mathcal{C}$. Under Hypothesis $\mathcal{F}$ every non-trivial zero with imaginary part in $[T,2T]$ lies in some maximal cluster.

It will be profitable to consider separately clusters of different sizes, and according to the length of the Dirichlet polynomial which detects many elements of the cluster. 

\begin{definition}\label{defn:RNH}
Given a parameter $H \ge 1$ and $\sigma \in [\frac{1}{2},1]$, we define $R_{N,H}(\sigma)$ to be the number of zeros $\rho = \beta+i\gamma$ such that:
\begin{enumerate}
\item $\beta \geq \sigma$.
\item $\gamma \in [T,2T]$.
\item There is a maximal cluster $\mathcal{C}$ containing $\rho$ with $|\mathcal{C}|\in[H,2H]$ and $\text{Im}(\rho') \in [T,2T]$ for every $\rho' \in \mathcal{C}$ such that at least $H(\log{T})^{-4}$ zeros $\rho'\in\mathcal{C}$ are detected by a Dirichlet polynomial of length $N$.
\item There is a half-isolated zero $\rho_0 = \beta_0 + i\gamma_0$ such that $\beta_0 \geq \sigma$ and
\begin{align*}
|\gamma - \gamma_0| \le |\mathcal{C}|(\log T)^5
\end{align*}
for every $\sigma+i\gamma \in \mathcal{C}$.
\end{enumerate}
\end{definition}
With the above notation introduced, we are now in a position to state the main proposition which reduces bounding $N(\sigma,T)$ to bounding $R_{N,H}(\sigma)$ for various values of $N$ and $H$.
%
%
%
%
\begin{proposition}\label{prop:ClustersReduction}
Assume Hypothesis $\mathcal{F}$. Let $\sigma \in [1/2,1]$ be such that there are non-trivial zeros with real part equal to $\sigma$. Then there is an absolute constant $c_\mathcal{F} > 0$ such that
\begin{align*}
N(\sigma,2T) - N(\sigma,T) &\ll T^{2(1-\sigma)}(\log T)^{O(1)} + N(\sigma+c_\mathcal{F},3T) (\log T)^{O(1)} \\
&+ (\log{T})^{O(1)}\sup_{\substack{T^{1/100} \leq N \leq T^{1/2}(\log T)^2 \\  1\leq H \ll T}}R_{N,H}(\sigma).
\end{align*}
\end{proposition}
%
%
%
%
Proposition \ref{prop:ClustersReduction} is the key place in the argument where we make use of the combinatorial structure imposed by Hypothesis $\mathcal{F}$. With this proposition established, our task is to obtain suitable bounds on $R_{N,H}(\sigma)$.
%
%
%
%
\begin{proposition}[Half-isolated zeros with clustering]\label{prop:HalfIsolatedClustering}
Assume Hypothesis $\mathcal{F}$, and let $\sigma \in [\frac{1}{2},1]$. Then
\[
R_{N,H}(\sigma)\ll T^{2(1-\sigma)}\exp((\log \log T)^{O(1)}).
\]
\end{proposition}
%
%
%
%

\begin{proof}[Proof of Theorem \ref{thm:ZeroDensity} assuming Propositions \ref{prop:ClustersReduction} and \ref{prop:HalfIsolatedClustering}]
The two propositions together imply
\begin{align*}
N(\sigma,2T) - N(\sigma,T) &\ll T^{2(1-\sigma)}\exp((\log \log T)^{O(1)}) + N(\sigma+c_\mathcal{F},3T) (\log T)^{O(1)}.
\end{align*}
Using this with $T/2,T/4,\ldots$ in place of $T$, we find
\begin{align}\label{eqn:Nsigma by Nsigma plus c}
N(\sigma,2T) &\ll T^{2(1-\sigma)}\exp((\log \log T)^{O(1)}) + N(\sigma+c_\mathcal{F},3T) (\log T)^{O(1)}.
\end{align}
By repeated application of \eqref{eqn:Nsigma by Nsigma plus c} we finally obtain
\begin{align*}
N(\sigma,2T) &\ll T^{2(1-\sigma)}\exp((\log \log T)^{O(1)}),
\end{align*}
noting that, under Hypothesis $\mathcal{F}$, we apply \eqref{eqn:Nsigma by Nsigma plus c} at most $O(1)$ times.
\end{proof}
%
%
%
%
\begin{proof}[Proof of Corollary \ref{cor:ShortIntervals} and Corollary \ref{cor:AlmostAllShortIntervals}]
Corollaries \ref{cor:ShortIntervals} and \ref{cor:AlmostAllShortIntervals} both follow from Theorem \ref{thm:ZeroDensity} via the explicit formula by standard techniques. Using the bound of Theorem \ref{thm:ZeroDensity} in the argument of \cite[Theorem 10.5]{IK2004} or \cite[Section 7.1]{Har2007} gives Corollary \ref{cor:ShortIntervals} immediately. For Corollary \ref{cor:AlmostAllShortIntervals} one needs to be slightly careful when $y\le x^{o(1)}$ in handling possible zeros very close to the line $\sigma=1$. By using the sharper bound $N(\sigma,T)\ll T^{167(1-\sigma)^{3/2}}(\log{T})^{17}$ (see \cite[Corollary 12.5]{Mon1971}) when $\sigma\ge 1-10^{-6}$ and Theorem \ref{thm:ZeroDensity} when $\sigma\le 1-10^{-6}$ in the standard argument as given in \cite[Section 9.1]{Har2007}, we obtain Corollary \ref{cor:AlmostAllShortIntervals}.
\end{proof}
%
%
%
%
\section{Proof of Proposition \ref{prop:ClustersReduction}--Combinatorics of zeros}
%
%
%
%
We review the classical set-up for detecting zeros of the Riemann zeta function. Since most of the technical details are somewhat standard but slightly different from other treatments, we have included full proofs in the appendix.
%
%
%
%
\begin{definition}[Type I/II zeros] Let $\rho=\beta+i\gamma$ be a non-trivial zero with $\gamma\in [T,2T]$.
\begin{enumerate}
\item We say $\rho$ is a `Type I zero' if
\begin{align}\label{eq:TypeICondition}
|D_N(\rho)|\geq \frac{1}{3\log{T}}
\end{align}
for some $N=2^j\in [T^{1/100},T^{1/2}(\log{T})^2]$, where $D_N(s)$ is given by \eqref{eq:DNDef}.
\item We say $\rho$ is a `Type II zero' if
\begin{align}\label{eq:Type II zero condition}
\Bigl|\frac{1}{2\pi i} \int_{(-\beta + 1/2)} T^{s/2} \Gamma(s) M(\rho+s)\zeta(\rho+s) ds\Bigr| \geq \frac{1}{3},
\end{align}
where 
\[
M(s) := \sum_{m \leq 2T^{1/100}} \frac{\mu(m)}{m^s}.
\]
\end{enumerate}
\end{definition}
%
%
%
%
Standard zero-detecting arguments (see \cite[Chapter 12]{Mon1971} or Appendix \ref{sec:ZeroDetection}) show the following.
%
%
%
%
\begin{lemma}\label{lmm:TypeIIIZeros}
Let $T$ be sufficently large. Then every non-trivial zero $\rho=\beta+i\gamma$ with $\gamma\in[T,2T]$ is either a Type I zero or a Type II zero (or both).
\end{lemma}
%
%
%
%
The strategy is to show that these conditions can only hold infrequently. Let $R_I(\sigma,T)$ denote the number of Type I zeros with $\beta \geq \sigma$ and $\gamma\in[T,2T]$, and let $R_{II}(\sigma,T)$ denote the number of Type II zeros with $\beta \geq \sigma$ and $\gamma\in[T,2T]$. The Type II zeros cause few problems.
%
%
%
%
\begin{lemma}\label{lem:TypeIIZeroBound}
We have
\begin{align*}
R_{II}(\sigma,T) \ll T^{2(1-\sigma)}(\log T)^{O(1)}.
\end{align*}
\end{lemma}
%
%
%
%
\begin{proof}
See Appendix \ref{sec:ZeroDetection}.
\end{proof}
%
%
%
%
We now turn our attention to bounding the number $R_I(\sigma,T)$ of Type I zeros. Under Hypothesis $\mathcal{F}$ the zeros of $\zeta(s)$ lie on finitely many vertical lines, say $\text{Re}(s)=\sigma_1,\sigma_2,\dots,\sigma_k$. For each vertical line, we put the zeros $\rho=\beta+i\gamma$ with $\gamma\in[T,2T]$ on the line into maximal clusters (recall Definition \ref{dfntn:Cluster}). 

We initially take the first cluster $\mathcal{C}_1=\{\rho_1\}$ to consist of the zero with smallest imaginary part, and then repeatedly add to $\mathcal{C}_1$ any zeros on the vertical line which are within $(\log{T})^3$ of a zero already included in this cluster until all zeros not included in $\mathcal{C}_1$ are at least $(\log{T})^3$ away from all elements on $\mathcal{C}_1$. We then repeat the process by initially taking $\mathcal{C}_2=\{\rho_2\}$ to consist of the zero of smallest imaginary part not in $\mathcal{C}_1$, and then repeatedly adding any zero within $(\log{T})^3$ of an element of $\mathcal{C}_2$. Repeating this procedure all zeros on each vertical line are put into a finite number of $(\log{T})^3$-separated maximal clusters. We do this separately for all vertical lines. 

(Equivalently, we can define a relation by saying $\rho_1\sim\rho_2$ if $\text{Re}(\rho_1)=\text{Re}(\rho_2)$ and $|\text{Im}(\rho_1)-\text{Im}(\rho_2)|\le (\log{T})^3$, and then extend this to an equivalence relation by transitivity. The equivalence classes are then our maximal clusters of zeros.)

We now separately treat clusters according to the following definition.
%
%
%
%
\begin{definition}
Let $\mathcal{C}$ be a maximal cluster of zeros on the line $\text{Re}(s)=\sigma$ which contains a zero having imaginary part in $[T,2T]$.
\begin{itemize}
\item (Mainly Type II zeros) We say $\mathcal{C}$ is a `Type A cluster' if at least $|\mathcal{C}|/2$ zeros in $\mathcal{C}$ are Type II zeros.
\item (Zeros past the endpoints) We say $\mathcal{C}$ is a `Type B cluster' if at least $|\mathcal{C}|/2$ zeros in $\mathcal{C}$ are Type I zeros, and there is some $\rho = \beta+i\gamma \in \mathcal{C}$ such that $\gamma < T$ or $\gamma > 2T$.
\item (Many nearby zeros to the right) We say $\mathcal{C}$ is a `Type C cluster' if it is not Type A or Type B and
\[
\Bigl|\Bigl\{\rho:\,\text{Re}(\rho)>\sigma,\,\min_{\rho'\in\mathcal{C}}|\rho-\rho'|\le (\log{T})^3|\mathcal{C}|\Bigr\}\Bigr|\ge \frac{|\mathcal{C}|}{(\log{T})^{100}}.
\]
\item (Mainly Type I zeros and few nearby zeros to the right) We say $\mathcal{C}$ is a `Type D cluster' if at least $|\mathcal{C}|/2$ zeros in $\mathcal{C}$ are Type I zeros, if $\mathcal{C}$ is not Type B, and
\[
\Bigl|\Bigl\{\rho:\,\text{Re}(\rho)>\sigma,\,\min_{\rho'\in\mathcal{C}}|\rho-\rho'|\le (\log{T})^3|\mathcal{C}|\Bigr\}\Bigr|< \frac{|\mathcal{C}|}{(\log{T})^{100}}.
\]
\end{itemize}
\end{definition}
%
%
%
%
It is easy to see that every maximal cluster is at least one of Type A, B, C, or D. Type A, Type B, and Type C clusters will be relatively straightforward to handle.
%
%
%
%
\begin{lemma}\label{lmm:TypeAClusters}
Let $\mathcal{C}_1,\dots,\mathcal{C}_J$ be the Type A clusters in the region $\textup{Re}(s)\geq\sigma$. Then
\[
\sum_{j=1}^J\sum_{\substack{\rho \in \mathcal{C}_j \\ \textup{Im}(\rho) \in [T,2T]}}1\ll T^{2(1-\sigma)}(\log T)^{O(1)}.
\]
\end{lemma}
%
%
%
%
\begin{proof}
Since the $\mathcal{C}_j$ are disjoint, and at least half of the elements are Type II zeros, we have that
\[
\sum_{j=1}^J\sum_{\substack{\rho \in \mathcal{C}_j \\ \textup{Im}(\rho) \in [T,2T]}}1\le 2R_{II}(\sigma,T).
\]
Lemma \ref{lem:TypeIIZeroBound} now gives the result.
\end{proof}
%
%
%
%
Before continuing on to Type B clusters we pause to state several auxiliary results, which will be used here and throughout the paper. The first result allows us to pass from discrete averages of Dirichlet polynomials to continuous averages.
%
%
%
%
\begin{lemma}\label{lem:DiscreteToContinuous}
Let $N$ be a large, positive real number. Let $\ell$ be a positive integer, and let $a_n$ be a sequence of complex numbers supported on $n \leq N$ which satisfies $|a_n| \leq N^A$. Then for $t \in \mathbb{R}$ one has the bound
\begin{align*}
\Bigl|\sum_{n \leq N} a_n n^{-it}\Bigr| \ll_A (\log N)\int_{|u|\leq (\log N)^2} \Bigl|\sum_{n \leq N} a_n n^{-it-iu}\Bigr|du + \exp (-(\log N)^{1.1}).
\end{align*}
\end{lemma}
\begin{proof}
This is a slight reformulation of \cite[Lemma 4.48]{Bou2000}.
\end{proof}
%
%
%
%
Next, we need a result which facilitates summing a Dirichlet polynomial over points with different real parts.
%
%
%
%
\begin{lemma}\label{lem:RealPartsBiggerThanSigma}
Let $z_1 = \sigma_1 + it_1, \ldots, z_r = \sigma_r + it_r$ be complex numbers with $\frac{1}{2}\leq \sigma \leq \sigma_k\leq 1$ for all $1\leq k\leq r$. Then for any finite sequence of complex numbers $a_n$ we have
\begin{align*}
\sum_{k=1}^r \Big|\sum_{n} \frac{a_n}{n^{\sigma_k + it_k }} \Big| \leq \int_0^{1/2}\sum_{k=1}^r \left(\Big|\sum_{n} \frac{2a_n}{n^{\sigma + it_k }} \Big| + \Big|\sum_{n} \frac{a_n}{n^{\sigma + it_k }} \frac{\log n}{n^\alpha} \Big| \right) d\alpha.
\end{align*}
\end{lemma}
%
%
%
%
\begin{proof}
We have
\begin{align*}
\frac{1}{n^{\sigma_k}} = \frac{1}{n^\sigma} - \frac{\log n}{n^\sigma}\int_0^{\sigma_k-\sigma} \frac{d\alpha}{n^\alpha},
\end{align*}
so by the triangle inequality we derive
\begin{align*}
\left|\sum_{n} \frac{a_n}{n^{\sigma_k + it_k }} \right|&\leq \left|\sum_{n} \frac{a_n}{n^{\sigma + it_k }} \right| + \int_0^{\sigma_k-\sigma} \left|\sum_{n} \frac{a_n}{n^{\sigma + it_k }} \frac{\log n}{n^\alpha} \right| d\alpha.
\end{align*}
Extending the integral by positivity and summing over $k$ then gives the result.
\end{proof}
%
%
%
%
\begin{remark}
Given a Dirichlet polynomial $F(s) = \sum_n f(n)n^{-s}$, we write $\tilde{F}(s) = \sum_n \tilde{f}(n) n^{-s}$ for a Dirichlet polynomial that results after applying Lemma \ref{lem:RealPartsBiggerThanSigma}. In particular $|\tilde{f}(n)| \ll \log(3n) |f(n)|$.
\end{remark}
%
%
%
%
Finally, we need a result establishing a relationship between $N$, the length of a zero-detecting polynomial, and $H$, the number of elements in a cluster.
%
%
%
%
\begin{lemma}\label{lmm:MaxClustersShort}
Let $\mathcal{C}$ be a maximal cluster with $|\mathcal{C}| = H\geq 1$. Assume $\textup{Re}(\rho) \geq \frac{1}{2} + \frac{1}{\log \log T}$, $\textup{Im}(\rho) \in [T/2,3T]$ for all $\rho \in \mathcal{C}$. Assume also that $\geq H(\log T)^{-O(1)}$ of the elements of $\mathcal{C}$ are detected by a Dirichlet polynomial of length $N$. If $T$ is sufficiently large, then $N \geq H (\log T)^{O(1)}$.
\end{lemma}
%
%
%
%
\begin{proof}
Assume for contradiction that $N \leq H(\log T)^{O(1)}$. Since $\geq H(\log T)^{-O(1)}$ elements of $\mathcal{C}$ are detected by a Dirichlet polynomial $D_N(s)$ of length $N\geq T^{1/100}$, we see that
\begin{align*}
\frac{H}{(\log T)^{O(1)}} \ll \sum_{\rho=\beta+it \in \mathcal{C}} \left|D_N(\beta+it) \right|^2.
\end{align*}
We apply Lemma \ref{lem:RealPartsBiggerThanSigma} with $\sigma = \frac{1}{2} + \frac{1}{\log \log T}$ and then take a supremum in $\alpha$ to obtain
\begin{align*}
\frac{H}{(\log T)^{O(1)}} \ll \sum_{\rho=\beta+it \in \mathcal{C}} \left|\widetilde{D}_N(\sigma+it) \right|^2.
\end{align*}
By Lemma \ref{lem:DiscreteToContinuous} we have
\begin{align*}
\frac{H}{(\log T)^{O(1)}} &\ll T^{-1} +  (\log T)^{O(1)} \sum_{\rho=\sigma+it \in \mathcal{C}} \int_{|u|\leq (\log T)^2}|\widetilde{D}_N(\sigma+it+iu)|^2 du \\
&\ll T^{-1} + (\log T)^{O(1)} \int_{-H(\log T)^{O(1)}}^{H(\log T)^{O(1)}}|\widetilde{D}_N(\sigma+iT_0 + iw)|^2 dw,
\end{align*}
for some $T_0 \asymp T$. The $T^{-1}$ term on the right-hand side may be deleted since $H \geq 1$. By the mean value theorem for Dirichlet polynomials \cite[Theorem 9.1]{IK2004} we then obtain
\begin{align*}
\frac{H}{(\log T)^{O(1)}} \ll (\log T)^{O(1)} H \sum_{N/2 < n \leq 2N} \frac{\tau(n)^2}{n^{2\sigma}} \ll (\log T)^{O(1)} H N^{-(2\sigma - 1)}\leq H T^{-(\log \log T)^{-3/2}},
\end{align*}
and this is a contradiction for $T$ sufficiently large.
\end{proof}
%
%
%
%
Now we show that the number of zeros in Type B clusters is small.
%
%
%
%
\begin{lemma}\label{lmm:exceptionalClusters}
Let $\mathcal{C}_1,\ldots,\mathcal{C}_J$ be the Type B clusters in the region $\textup{Re}(s)\geq\sigma$. If $\sigma \in [\frac{1}{2} + \frac{1}{\log \log T}, 1]$ then
\begin{align*}
\sum_{j=1}^J |\mathcal{C}_j| \ll T^{2(1-\sigma)}(\log T)^{O(1)}.
\end{align*}
\end{lemma}
%
%
%
%
\begin{proof}
By the definition of a Type B cluster and Lemma \ref{lmm:MaxClustersShort} we see that the imaginary parts of elements of $\bigcup_{j=1}^J \mathcal{C}_j$ are contained in
\begin{align*}
[T - T^{1/2}(\log T)^{O(1)},T + T^{1/2}(\log T)^{O(1)}] \cup [2T - T^{1/2}(\log T)^{O(1)},2T + T^{1/2}(\log T)^{O(1)}].
\end{align*}
Write $R = \bigcup_{j=1}^J |\mathcal{C}_j|$. By the pigeonhole principle there is some $N \in [T^{1/100},T^{1/2}(\log T)^2]$ such that $\gg R/\log T$ of the elements of $\bigcup_{j=1}^J \mathcal{C}_j$ are detected by a Dirichlet polynomial of length $N$. Let $1\leq k\leq 100$ be such that $N^k \leq T^{1/2}(\log T)^{O(1)}\leq N^{k+1}$. By Lemma \ref{lem:RealPartsBiggerThanSigma} and then Lemma \ref{lem:DiscreteToContinuous} we have
\begin{align*}
R &\ll (\log T)^{O(1)}\sum_{\ell=1}^2 \int_{\ell T - T^{1/2}(\log T)^{O(1)}}^{\ell T + T^{1/2}(\log T)^{O(1)}} |\widetilde{D}_N(\sigma+it)|^{2k} dt,
\end{align*}
and the mean-value theorem for Dirichlet polynomials \cite[Theorem 9.1]{IK2004} yields
\begin{align*}
R &\ll T^{1/2}(\log T)^{O(1)} N^{-k(2\sigma-1)} \ll T^{1-\sigma + \frac{2\sigma-1}{2k+2}} (\log T)^{O(1)}.
\end{align*}
Since $k\geq 1$ we deduce
\begin{align*}
T^{1-\sigma + \frac{2\sigma-1}{2k+2}} (\log T)^{O(1)} \leq T^{2(1-\sigma)}(\log T)^{O(1)}
\end{align*}
for $\sigma \in [\frac{1}{2},\frac{5}{6}]$. Work of Huxley \cite{Hux1972} gives the result when $\sigma \geq \frac{5}{6}$.
\end{proof}
%
%
%
%
\begin{lemma}\label{lmm:TypeCClusters}
Assume Hypothesis $\mathcal{F}$. Let $\mathcal{C}_1,\dots,\mathcal{C}_J$ be the Type C clusters on $\textup{Re}(s)=\sigma$. Then
\[
\sum_{j=1}^J|\mathcal{C}_j|\ll (\log{T})^{105}\Bigl(N(\sigma+c_\mathcal{F},3T)-N(\sigma+c_\mathcal{F},T/2)\Bigr).
\]
\end{lemma}
%
%
%
%
\begin{proof}
By Hypothesis $\mathcal{F}$ any zero $\rho$ with $\text{Re}(\rho)>\sigma$ has $\text{Re}(\rho)\ge \sigma+c_\mathcal{F}$. Thus, if $\mathcal{C}$ is a Type C cluster on $\text{Re}(s)=\sigma$ with $|\mathcal{C}|\in[L,2L]$, then we see 
\[
|\mathcal{C}|\le (\log{T})^{100}\sum_{\substack{\text{Re}(\rho)\ge \sigma+c_{\mathcal{F}}\\ \min_{\rho'\in\mathcal{C}}|\rho-\rho'|\le 2L(\log{T})^3}}1.
\]
By Lemma \ref{lmm:MaxClustersShort} we have $L\leq T^{1/2}(\log T)^{O(1)}$. Summing over all such $\mathcal{C}$ and swapping the order of summation, we see 
\begin{align*}
\sum_{\substack{ |\mathcal{C}_j|\in[L,2L]}}|\mathcal{C}_j|&\le (\log{T})^{100}\sum_{\substack{\text{Re}(\rho)\ge \sigma+c_\mathcal{F} \\ T-2L(\log T)^3 \leq \text{Im}(\rho)\leq 2T + 2L(\log T)^3}}\sum_{\substack{\mathcal{C}\text{ maximal}\\ |\mathcal{C}|\ge L\\ |\rho-\mathcal{C}|\le 2L(\log{T})^3}}1\\
&\ll (\log{T})^{104}\Bigl(N(\sigma+c_\mathcal{F},3T)-N(\sigma+c_\mathcal{F},T/2)\Bigr).
\end{align*}
Summing over the $O(\log{T})$ possibilities for $L$ a power of two gives the result.
\end{proof}
%
%
%
%
\begin{lemma}\label{lmm:NearbyHalfIsolated}
Let $\mathcal{C}$ be a Type D cluster on $\textup{Re}(s)=\sigma$. Then there is a half-isolated zero $\rho_0$ such that $\textup{Re}(\rho_0)\ge \sigma$ and $\min_{\rho'\in\mathcal{C}}|\rho_0-\rho'|\le (\log{T})^3|\mathcal{C}|$.
\end{lemma}
%
%
%
%
\begin{proof}
We construct a sequence of zeros. Let $\rho_1\in\mathcal{C}$ be the zero with smallest imaginary part. Given $\rho_j=\beta_j+i\gamma_j$, let
\[
\mathcal{Z}_j:=\{\rho=\beta+i\gamma:\,0<|\rho-\rho_j|\le 2(\log{T})^2,\,\beta>\beta_j\text{ or }\beta=\beta_j\text{ and }\gamma< \gamma_j\}. 
\]
If $\mathcal{Z}_j$ is empty, then $\rho_j$ is half-isolated. If $\mathcal{Z}_j$ is non-empty, we take $\rho_{j+1}$ to be an element of $\mathcal{Z}_j$ with largest real part. We note that $|\rho_{i+1}-\rho_i|\le 2(\log{T})^2$ for each $i$, so $\min_{\rho'\in\mathcal{C}}|\rho_j-\rho'|\le 2j(\log{T})^2$. Moreover, either $\beta_{j+1}=\beta_j$ and $\gamma_{j+1}<\gamma_j$ or $\beta_{j+1}>\beta_j$, so these zeros are all distinct. Since $\mathcal{C}$ is a cluster and $\rho_1$ has smallest real part, either $\rho_1$ is half-isolated or $\beta_2>\beta_1$. Therefore if there are no half-isolated zeros within $(\log{T})^3|\mathcal{C}|$ of $\mathcal{C}$, this procedure will produce a sequence of at least $|\mathcal{C}|$ zeros within $(\log{T})^3|\mathcal{C}|$ of $\mathcal{C}$ and all to the right of $\text{Re}(s)=\sigma$. But this contradicts the assumption that $\mathcal{C}$ is of Type D.
\end{proof}
%
%
%
%
\begin{lemma}\label{lmm:TypeDSubset}
Let $\mathcal{C}$ be a Type D cluster on $\textup{Re}(s)=\sigma$. Then there is a subset $\mathcal{C}'\subseteq\mathcal{C}$ and $N=2^j\in[T^{1/100},T^{1/2}(\log{T})^2]$ such that all of the following hold:
\begin{itemize}
\item ($\mathcal{C}'$ is not too small)
\[
|\mathcal{C}'|\ge|\mathcal{C}|/(\log{T})^2.
\]
\item (There is a nearby half-isolated zero) There is a half-isolated zero $\rho_0$ such that $\textup{Re}(\rho_0)\ge \sigma$ and $|\rho_0-\rho|\le (\log{T})^5|\mathcal{C}'|$ for all $\rho\in\mathcal{C}'$.
\item (All zeros in $\mathcal{C}'$ are detected by a polynomial of length $N$) If $\rho\in\mathcal{C}'$ then
\[
\Bigl|\sum_{n\sim N} \frac{a(n)}{n^\rho}\exp \Bigl( - \frac{n}{T^{1/2}}\Bigr) \Bigr| \gg \frac{1}{\log T}.
\]
\end{itemize}
\end{lemma}
%
%
%
%
\begin{proof}
Let $\mathcal{C}_1$ be the set of Type I zeros in $\mathcal{C}$. Since $\mathcal{C}$ is a Type D cluster, $|\mathcal{C}_1|\ge |\mathcal{C}|/2$. From  \eqref{eq:TypeICondition} every Type I zero satisfies
\[
\Bigl|\sum_{n\sim M} \frac{a(n)}{n^\rho} \exp \Bigl( - \frac{n}{T^{1/2}}\Bigr) \Bigr| \gg \frac{1}{\log T},
\]
for some $M=2^j\in [T^{1/100},T^{1/2}(\log{T})^2]$. By the pigeonhole principle there is some $N=2^j$ such that at least $|\mathcal{C}_1|/\log{T}$ zeros satisfy this with $M=N$. We let $\mathcal{C}'$ be this set of zeros, so $|\mathcal{C}'|\gg|\mathcal{C}|/\log{T}$. It follows from Lemma \ref{lmm:NearbyHalfIsolated} that there is a half-isolated zero within $(\log{T})^3|\mathcal{C}|$ of $\mathcal{C}$ to the right of $\text{Re}(s)=\sigma$. Since $\mathcal{C}$ is contained in an interval of length $(\log{T})^3|\mathcal{C}|$ we see that $\rho_0$ must satisfy
\[
|\rho_0-\mathcal{C}'|\le 2(\log{T})^3|\mathcal{C}|\ll (\log{T})^4|\mathcal{C}'|.
\]
This gives the result.
\end{proof}
%
%
%
%
We are now in a position to prove Proposition \ref{prop:ClustersReduction}.
%
%
%
%
\begin{proof}[Proof of Proposition \ref{prop:ClustersReduction}]
We may assume $\sigma \geq \frac{1}{2} + c$ for some absolute constant $c>0$, since otherwise the result is trivial by Hypothesis $\mathcal{F}$.

By Hypothesis $\mathcal{F}$, all zeros have real part in a finite set $\{\sigma_1,\dots,\sigma_J\}$ with $\sigma_{j+1}>\sigma_j+c_{\mathcal{F}}$ for some  absolute constant $c_{\mathcal{F}}>0$. By the above discussion, every non-trivial zero $\rho=\beta+i\gamma$ with $\gamma\in[T,2T]$ lies in a cluster of Type A, Type B, Type C, or Type D. Therefore
\[
N(\sigma,2T)-N(\sigma,T)\le \sum_{\sigma_j\ge \sigma}\Bigl(N_A(\sigma_j)+N_B(\sigma_j)+N_C(\sigma_j) + N_D(\sigma_j)\Bigr),
\]
where $N_A(\sigma)$ counts the number of zeros $\rho=\beta+i\gamma$ with $\beta=\sigma$ and $\gamma\in[T,2T]$ which lie on a cluster of Type A, and similarly for $N_B(\sigma)$, $N_C(\sigma),N_D(\sigma)$. By Lemmas \ref{lmm:TypeAClusters}, \ref{lmm:exceptionalClusters}, and \ref{lmm:TypeCClusters} and we have 
\begin{align*}
N_A(\sigma),N_B(\sigma)&\le T^{2(1-\sigma)}(\log{T})^{O(1)},\\
N_C(\sigma)&\le (\log{T})^{O(1)}N(\sigma+c_{\mathcal{F}},3T).
\end{align*}
 We split the count $N_D(\sigma)$ by dyadically pigeonholing on the length of the cluster. This gives
 \begin{equation}
 N(\sigma,2T)-N(\sigma,T)\le (\log{T})^{O(1)}T^{2(1-\sigma)}+(\log{T})^{O(1)}N(\sigma+c_{\mathcal{F}},3T)+\sum_{\sigma_j\ge \sigma}\sum_{\substack{H=2^j\\ 1\leq H\ll T}}N_{D}(\sigma_j;H),
 \label{eq:NBound2}
 \end{equation}
 where $N_{D}(\sigma;H)$ counts the total number of zeros $\rho=\sigma+i\gamma$ with $\gamma\in[T,2T]$ in a cluster of Type D which has cardinality between $H$ and $2H$. By Lemma \ref{lmm:TypeDSubset} every Type D cluster has a large subset $\mathcal{C}'$ such that all elements of $\mathcal{C}'$ are detected by a Dirichlet polynomial of length $N=2^j\in [T^{1/100},T^{1/2}(\log{T})^2]$. We split $N_{D}(\sigma;L)$ according to this quantity $N$, and note that these zeros are all counted by $R_{N,H}(\sigma)$ (by Lemma \ref{lmm:TypeDSubset}). Thus
 \[
 \sum_{\sigma_j\ge \sigma}N_{C}(\sigma_j;L)\le (\log T)^{O(1)} \sum_{\substack{N=2^j\\ T^{1/100}\le N\le T^{1/2}(\log{T})^2}}R_{N,H}(\sigma).
 \]
 Substituting this into our bound \eqref{eq:NBound2} for $N(\sigma,2T)-N(\sigma,T)$ and noting that there are $O(\log{T})^2$ different terms $L,N$ in the summation then gives the result. 
\end{proof}
%
%
%
%
\section{Proof of Proposition \ref{prop:HalfIsolatedClustering}--Half-isolated zeros with clustering}\label{sec:ClusteringBound}
%
%
%
%
We prove an auxiliary result, which says that, under Hypothesis $\mathcal{F}$, half-isolated zeros have zero-detecting polynomials of flexible length.
%
%
%
%
\begin{lemma}[Half-isolated zero detector of flexible length]\label{lem:GoodLengthForHIZeros}
Assume Hypothesis $\mathcal{F}$, and let $T$ be sufficiently large. There is a set $\mathcal{D} = \mathcal{D}(T)$ of Dirichlet polynomials with $\#\mathcal{D} \leq \exp((\log \log T)^3)$ such that the following holds.

For any half-isolated zero $\rho_0 = \sigma+i\gamma$, with $\gamma \in [T,2T]$ and any choice of $A$ with $\log{A}\in[(\log \log T)^4 ,\log T]$ there is some $B \in \mathcal{D}$ such that
\begin{itemize}
\item $|B(\rho_0)| \geq \exp(-(\log \log T)^3)$, and
\item $B(s) = \sum_m b(m) m^{-s}$ with $|b(m)| \leq \exp((\log \log T)^3)$ and $b(m)$ supported on
\begin{align*}
m \in [A\exp(-2(\log \log T)^4),A\exp((\log \log T)^2)].
\end{align*}
\end{itemize}
\end{lemma}
\begin{proof}
By Hypothesis $\mathcal{F}$ and the definition of a $Y$-half-isolated zero, we deduce that $\rho_0$ is $Y$-half-isolated for any $Y \in [\exp((\log \log T)^3),T]$. Set $A_1 = A$, and note that $\rho_0$ is $A_1^{1/2}$-half-isolated. Proposition \ref{prp:HalfIsolated} then implies there is some $U_1 \in \mathcal{U} \cap (A_1^{1/2},A_1]$ such that
\begin{align*}
\Big|\sum_n \frac{\Lambda(n)}{n^{\rho_0}}w_0(n/U_1) \Big| \gg (\log T)^{-100}.
\end{align*}
If $A_1/U_1\leq \exp((\log \log T)^4)$ then we stop. Otherwise, we set $A_2 = A_1/U_1$ and repeat the process with $A_1$ replaced by $A_2$.

Continuing in this fashion, we obtain sequences $A_1,A_2,\ldots,A_k$ and $U_1,U_2,\ldots,U_k$ such that $U_i \in (A_i^{1/2},A_i]$ for all $1\leq i \leq k$, $A_i = A_{i-1}/U_{i-1}$ for $2\leq i \leq k$, $A_i > \exp((\log \log T)^4)$ for $i \leq k-1$ and $A_k \leq \exp((\log \log T)^4)$. Furthermore, we have
\begin{align*}
\Big|\sum_n \frac{\Lambda(n)}{n^{\rho_0}}w_0(n/U_i) \Big| \gg (\log T)^{-100}
\end{align*}
for $1\leq i \leq k-1$. By induction we find that $A_i \leq A_1^{2^{-i+1}} \leq T^{2^{-i+1}}$, and since $A_{k-1} > \exp((\log \log T)^4)$ we deduce $k\ll \log \log T$. By construction we have $U_1 U_2 \cdots U_{k-1}\leq A_1 = A$, and
\begin{align*}
\exp((\log \log T)^4) \geq A_k = \frac{A_{k-1}}{U_{k-1}} = \frac{A_{k-2}}{U_{k-1}U_{k-2}} = \cdots = \frac{A_1}{U_{k-1}U_{k-2}\cdots U_1} = \frac{A}{U_{k-1}U_{k-2}\cdots U_1},
\end{align*}
so
\begin{align*}
A \exp(-(\log \log T)^4) \leq U_1 U_2 \cdots U_{k-1} \leq A.
\end{align*}

We define a Dirichlet polynomial $B(s)$ by
\begin{align*}
B(s) &= \sum_m \frac{b(m)}{m^s} = \prod_{i=1}^{k-1}\sum_{n_i} \frac{\Lambda(n_i)}{n_i^s}w_0(n_i/U_i),
\end{align*}
and note that $|B(\rho_0)| \geq \exp(-(\log \log T)^3)$. The coefficients $b(m)$ are supported on integers $m$ with
\begin{align*}
A\exp(-2(\log \log T)^4)\leq \frac{U_1 \cdots U_{k-1}}{2^{k-1}}\leq m\leq 2^{k-1}U_1\cdots U_{k-1} \leq A \exp((\log \log T)^2).
\end{align*}
Furthermore, we have
\begin{align*}
|b(m)| &\leq \sum_{\substack{n_1\cdots n_{k-1} = m \\ U_i/2 \leq n_i \leq 2U_i}} \Lambda(n_1)\cdots \Lambda(n_{k-1})\leq (2\log m)^{k-1}(\log 2T)^{k-1}\leq \exp((\log \log T)^3).
\end{align*}
Here we noted that if $m=p_1^{e_1}\cdots p_j^{e_j}$ then the number of choices of prime powers dividing $n$ is $\sum_{i=1}^je_i\leq 2\log m$, and so there are $\leq (2\log m)^{k-1}$ choices of $n_1,\dots,n_{k-1}$.

The bound for $\#\mathcal{D}$ arises from the trivial bound $\#\mathcal{D} \leq |\mathcal{U}|^{O(\log \log T)}$.
\end{proof}
%
%
%
%
\begin{proof}[Proof of Proposition \ref{prop:HalfIsolatedClustering}]
We first show that we may assume $N$ is not too near $T^{1/2}$. Indeed, assume that $N \geq T^{1/2}\exp(-(\log \log T)^{10})$. Then by the definition of $R_{N,H}(\sigma)$ (Definition \ref{defn:RNH}) we have
\begin{align*}
R_{N,H}(\sigma) &\ll (\log T)^{O(1)} \sum_{\rho = \beta + i\gamma} |D_N(\beta+i\gamma)|^4
\end{align*}
for some zeros $\rho = \beta+i\gamma$ with $\beta \geq \sigma$ and $\gamma \in [T,2T]$. By Lemmas \ref{lem:RealPartsBiggerThanSigma} and \ref{lem:DiscreteToContinuous} we obtain
\begin{align*}
R_{N,H}(\sigma) &\ll (\log T)^{O(1)} \int_{T/2}^{3T}|\tilde{D}_N(\sigma+iw)|^4 dw,
\end{align*}
and the mean-value theorem for Dirichlet polynomials yields
\begin{align*}
R_{N,H}(\sigma) &\ll (\log T)^{O(1)} T N^{-4\sigma+2} \leq T^{2(1-\sigma)} \exp((\log \log T)^{O(1)}).
\end{align*}

Therefore, we may henceforth assume $N \leq T^{1/2}\exp(-(\log \log T)^{10})$. By Hypothesis $\mathcal{F}$ we may also assume $\sigma \geq \frac{1}{2} + c$ for some absolute constant $c>0$, for otherwise the result is trivial. Further, we may assume that $N\geq H(\log T)^{O(1)}$, since Lemma \ref{lmm:MaxClustersShort} implies $R_{N,H}(\sigma)=0$ if $N\leq H(\log T)^{O(1)}$.

Consider a cluster $\mathcal{C}$ of zeros of length $\asymp H$ which is considered by $R_{N,H}(\sigma)$. A proportion $\gg (\log T)^{-O(1)}$ of the elements of $\mathcal{C}$ are detected by a Dirichlet polynomial of length $N$, and therefore
\begin{align*}
H(\log T)^{-O(1)} \ll \sum_{\rho \in \mathcal{C}} |D_N(\rho)|^2.
\end{align*}
To avoid technical complications later in the proof, it is convenient now to decompose $D(\rho)$ into level sets according to the size of the divisor function. We write
\begin{align*}
D_N(\rho) &= \sum_{n\sim N} \frac{a(n)}{n^\rho}\exp(-n/T^{1/2}) = \sum_{1\ll V \ll N^{o(1)}} \sum_{\substack{n\sim N \\ \tau(n) \sim V}} \frac{a(n)}{n^\rho}\exp(-n/T^{1/2}) = \sum_{1\ll V \ll N^{o(1)}} D_{N,V}(\rho),
\end{align*}
say. We then apply Cauchy-Schwarz to deduce
\begin{align*}
H(\log T)^{-O(1)} \ll \sum_V\sum_{\rho \in \mathcal{C}} |D_{N,V}(\rho)|^2.
\end{align*}

We may still use Lemma \ref{lem:RealPartsBiggerThanSigma} to remove the dependence on the real parts of the zeros $\rho$, but we must proceed more carefully than we have heretofore. Lemma \ref{lem:RealPartsBiggerThanSigma} implies
\begin{align*}
H(\log T)^{-O(1)} \ll \int_0^{1/2}\sum_{j=1,2}\sum_V\sum_{\rho \in \mathcal{C}} |\tilde{D}_{j,N,V}(\sigma+i\gamma)|^2 d\alpha,
\end{align*}
where the Dirichlet polynomials $\tilde{D}_{j,N,V}$ depend on $\alpha$, but we do not denote this in the notation. By the definition of $R_{N,H}(\sigma)$, there is a half-isolated zero $\rho_0 = \beta_0 + i\gamma_0$ such that $|\gamma-\gamma_0| \leq H(\log T)^{O(1)}$ for every $\rho = \beta+i\gamma \in \mathcal{C}$. By Lemma \ref{lem:DiscreteToContinuous} we then find
\begin{align*}
H(\log T)^{-O(1)} \ll (\log T)^{O(1)}\int_0^{1/2}\sum_{j=1,2}\sum_V \int_{\gamma_0-H'}^{\gamma_0+H'} |\tilde{D}_{j,N,V}(\sigma+it)|^2 dt d\alpha,
\end{align*}
where $H' = H(\log T)^{O(1)}$. Now let $f$ be a smooth, non-negative function with Fourier transform supported in $|y| \leq 1$ with $f(x) \gg 1$ for $|x| \leq 1$. We deduce
\begin{align*}
H(\log T)^{-O(1)} &\ll\int_0^{1/2}\sum_j\sum_V \int_\mathbb{R} f\left(\frac{t}{H'}\right)|\tilde{D}_{j,N,V}(\sigma+i\gamma_0 + it)|^2 dt \\
&= H'\int_0^{1/2}\sum_{j=1,2}\sum_V\mathop{\sum\sum}_{\substack{n_1,n_2 \sim N \\ \tau(n_1),\tau(n_2) \sim V}} \frac{\tilde{a}_j(n_1)\tilde{a}_j(n_2)}{(n_1n_2)^\sigma}\left(\frac{n_2}{n_1}\right)^{i\gamma_0} \widehat{f}\left(\frac{H'}{2\pi}\log(n_2/n_1)\right)d\alpha.
\end{align*}
The contribution from $n_1=n_2$ to the right-hand side is $\ll H' T^{-c'}$ for some constant $c'>0$, and this is smaller than the lower bound $H(\log T)^{-O(1)}$. We may therefore ignore this diagonal contribution. For the other terms, we note that $|\log(n_2/n_1)| \ll 1/H'$ by the support of $\widehat{f}$, and this implies $|n_2-n_1| \ll N/H' \ll N/H$. If we define
\begin{align*}
S(\gamma) &= \int_0^{1/2}\sum_{j=1,2}\sum_V\mathop{\sum\sum}_{\substack{\substack{n_1,n_2 \sim N\\ 1\leq |n_2-n_1| \ll N/H \\ \tau(n_1),\tau(n_2) \sim V}}} \frac{\tilde{a}_j(n_1)\tilde{a}_j(n_2)}{(n_1n_2)^\sigma}\left(\frac{n_2}{n_1}\right)^{i\gamma} \widehat{f}\left(\frac{H'}{2\pi}\log(n_2/n_1)\right) d\alpha\\ 
&= \int_0^{1/2}\sum_{j=1,2}\sum_V S_{j,V}(\gamma)d\alpha
\end{align*}
then we have $|S(\gamma_0)| \gg (\log T)^{-O(1)}$ for every half-isolated zero $\rho_0 = \beta_0 + i\gamma_0$ which is associated to a cluster of length $\asymp H$ as in the definition of $R_{N,H}(\sigma)$.

We require zero-detecting polynomials for the half-isolated zeros of a convenient length. We apply Lemma \ref{lem:GoodLengthForHIZeros} with $A = T N^{-2} \geq \exp((\log \log T)^{10})$ to obtain a Dirichlet polynomial $B(s)$ with the properties stated there. We define
\begin{align*}
K &= \sum_{\rho_0} |B(\rho_0)|^2 |S(\gamma_0)|^2,
\end{align*}
where the sum is over half-isolated zeros $\rho_0 = \beta_0 + i\gamma_0$ as in the definition of $R_{N,H}(\sigma)$. It follows that
\begin{align}\label{eqn:BoundRNHbyK}
R_{N,H}(\sigma) &\ll \exp((\log \log T)^{O(1)}) H K,
\end{align}
since we are counting zeros in clusters of length $\asymp H$. We take the supremum in $\alpha$, then apply Cauchy-Schwarz and the pigeonhole principle to obtain
\begin{align*}
K &\ll (\log T)^{O(1)} \sum_{j=1,2}\sum_V\sum_{\rho_0} |B(\rho_0)|^2 |S_{j,V}(\gamma_0)|^2 \\ 
&\ll (\log T)^{O(1)} \sum_{\rho_0} |B(\rho_0)|^2 |S_{j,V}(\gamma_0)|^2 = (\log T)^{O(1)} K_V
\end{align*}
for some $1\ll V \ll N^{o(1)}$ and $j\in\{1,2\}$. By dyadic subdivision and Cauchy-Schwarz, we may assume $B(s)$ is supported on integers $m \sim M$, where
\begin{align*}
\frac{T}{N^2}\exp(-(\log \log T)^{5}) \leq M \leq \frac{T}{N^2} \exp((\log \log T)^{5}).
\end{align*}
We apply Lemma \ref{lem:RealPartsBiggerThanSigma} to $B$ to remove the dependency on the real parts of $\rho_0$, obtaining
\begin{align*}
K_V &\ll \sum_{\rho_0} |\tilde{B}(\sigma+i\gamma_0)|^2 |S_{j,V}(\gamma_0)|^2.
\end{align*}
By a minor variant of Lemma \ref{lem:DiscreteToContinuous} we find
\begin{align*}
|\tilde{B}(\sigma+i\gamma_0)|^2 |S_V(\gamma_0)|^2 &\ll (\log T)^{O(1)} \int_{|u|\leq (\log T)^{O(1)}} |\tilde{B}(\sigma+i\gamma_0+iu)|^2 |S_{j,V}(\gamma_0+iu)|^2 du,
\end{align*}
and therefore
\begin{align*}
K_V &\ll (\log T)^{O(1)} \int_{T/2}^{3T}|\tilde{B}(\sigma+it)|^2 |S_{j,V}(t)|^2 dt.
\end{align*}

Let $\Psi$ be a suitable non-negative smooth function whose Fourier transform is supported in $|y|\leq 1$. Smoothing by $\Psi$, opening the squares, interchanging, and applying the triangle inequality, we obtain
\begin{align*}
K_V &\ll (\log T)^{O(1)}\int_\mathbb{R}\Psi(t/T)|\tilde{B}(\sigma+it)|^2 |S_{j,V}(t)|^2 dt \\
&\ll T \mathop{\sum\cdots \sum}_{\substack{m_1,m_2 \sim M \\ n_1,n_2,n_3,n_4 \sim N \\ 1\leq |n_2 - n_1| \ll N/H \\ 1\leq |n_4-n_3| \ll N/H \\ \tau(n_i) \sim V \\ |m_2n_2n_3 - m_1n_1n_4| \ll MN^2/T}} \frac{|\tilde{b}(m_1)\tilde{b}(m_2)|}{(m_1m_2)^\sigma} \frac{|\tilde{a}_j(n_1)\tilde{a}_j(n_2)\tilde{a}_j(n_3)\tilde{a}_j(n_4)|}{(n_1n_2n_3n_4)^\sigma} \\
&\ll \exp((\log \log T)^{O(1)})V^4 T M^{-2\sigma}N^{-4\sigma}\mathop{\sum\cdots \sum}_{\substack{m_1,m_2 \sim M \\ n_1,n_2,n_3,n_4 \sim N \\ 1\leq |n_2 - n_1| \ll N/H \\ 1\leq |n_4-n_3| \ll N/H \\ \tau(n_i) \sim V \\ |m_2n_2n_3 - m_1n_1n_4| \ll MN^2/T}} 1.
\end{align*}
Here we have used the bound on the coefficients $b(m)$ from Lemma \ref{lem:GoodLengthForHIZeros} and the fact that $|\tilde{a}_j(n)|\leq\tau(n) \sim V$. We drop the condition $1\leq |n_4-n_3| \ll N/H$, and then change variables $w_1 = m_1n_4, w_2 = m_2n_3$. By the size of $M$ we obtain
\begin{align*}
K_V &\ll \mathcal{L}V^4 T^{1-2\sigma} \mathop{\sum\sum}_{\substack{n_1,n_2 \sim N \\ 1\leq |n_2-n_1| \ll N/H \\ \tau(n_i) \sim V}} \mathop{\sum\sum}_{\substack{w_1,w_2 \sim MN \\ |n_1w_1 - n_2w_2| \leq \mathcal{L}}} \tau(w_1)\tau(w_2),
\end{align*}
where we have written $\mathcal{L} = \exp((\log \log T)^{O(1)})$. We apply the inequality
\begin{align*}
\tau(w_1)\tau(w_2) \leq \frac{1}{2}(\tau(w_1)^2 + \tau(w_2)^2)
\end{align*}
and use symmetry to deduce
\begin{align*}
K_V &\ll \mathcal{L} V^4 T^{1-2\sigma} \mathop{\sum\sum}_{\substack{n_1,n_2 \sim N \\ 1\leq |n_2-n_1| \ll N/H \\ \tau(n_i) \sim V}} \mathop{\sum\sum}_{\substack{w_1,w_2 \sim MN \\ |n_1w_1 - n_2w_2| \leq \mathcal{L}}} \tau(w_2)^2.
\end{align*}
We rewrite $n_1$ as $n$ and $w_2$ as $w$, and then change variables $n_2 = n+r, w_1 = w+s$, to obtain
\begin{align*}
K_V &\ll \mathcal{L} V^4 T^{1-2\sigma}\sum_{|\ell|\leq \mathcal{L}} \mathop{\sum\sum}_{\substack{n \sim N \\ 1\leq |r| \ll N/H \\ \tau(n) \sim V}} \mathop{\sum\sum}_{\substack{w \sim MN \\ |s|\ll MN \\ ns - rw =\ell}}\tau(w)^2.
\end{align*}
Observe that $s \neq 0$ since $r\neq 0$, and similarly $ns-\ell \neq 0$. The condition $|ns-rw| \leq \mathcal{L}$ implies $n|s| \ll MN^2/H$, and therefore $|s| \ll MN/H$. Since $w \mid ns-\ell$, it follows that
\begin{align*}
K_V &\ll \mathcal{L} V^4 T^{1-2\sigma}\sum_{|\ell|\leq\mathcal{L}} \sum_{\substack{n\sim N \\ \tau(n)\sim V}} \sum_{1\leq |s| \ll MN/H} \tau(|ns-\ell|)^3.
\end{align*}
We multiply by $\tau(n)^4 V^{-4} > 1$ and then change variables $f = n|s|$ to obtain
\begin{align*}
K_V &\ll \mathcal{L}  T^{1-2\sigma}\sum_{|\ell| \leq \mathcal{L}}\sum_{N \ll f\ll MN^2/H}\tau(f)^5 \tau(f+\ell)^3 \\
&\ll \mathcal{L}^3 T^{1-2\sigma} MN^2/H,
\end{align*}
the last inequality following by Cauchy-Schwarz and standard estimates for the divisor function. The constraints on $M$ imply $K_V \ll \mathcal{L}^4 T^{2-2\sigma} H^{-1}$, and combining with \eqref{eqn:BoundRNHbyK} yields the result.
\end{proof}
%
%
%
%
\section{Obstructions}\label{sec:Bows}
%
%
%
%
We hope that the ideas we have introduced in this paper can lead to new and improved zero-density estimates. However, as alluded to in the introduction, there are potential obstructions one must rule out in order to have a chance of making our results unconditional. We call these potential bad configurations ``bows'' of zeros. 

A bow of zeros is a set of zeros where the imaginary parts lie in an arithmetic progression with common difference $\approx \frac{1}{\log T}$, and where the real parts of the zeros vary smoothly between $\frac{1}{2}$ and some $\sigma > \frac{1}{2}$ and then back to $\frac{1}{2}$ again. An example is if the sequence $(x_j)_{j\le T^\epsilon}$ of complex numbers coincided with consecutive zeros of $\zeta(s)$, where $x_j$ is given by
\begin{align}\label{eq:BowOfZeros}
x_j=
\begin{cases}
\frac{1}{2} + \frac{j}{T^\varepsilon} + i\frac{c j}{\log T}+iT_0, \ \ \ \ \ \ &1\leq j\leq  \frac{1}{4}T^\varepsilon, \\
\frac{3}{4} + i\frac{c j }{\log T}+iT_0, &\frac{1}{4} T^\varepsilon < j \leq \frac{3}{4}T^\varepsilon, \\
\frac{3}{2} - \frac{j}{T^\varepsilon} + i\frac{c j}{\log T}+iT_0, &\frac{3}{4}T^\varepsilon < j \leq T^\varepsilon.
\end{cases}
\end{align}
Here $c>0$ is an absolute constant and $T_0 \asymp T$.
 \begin{center}
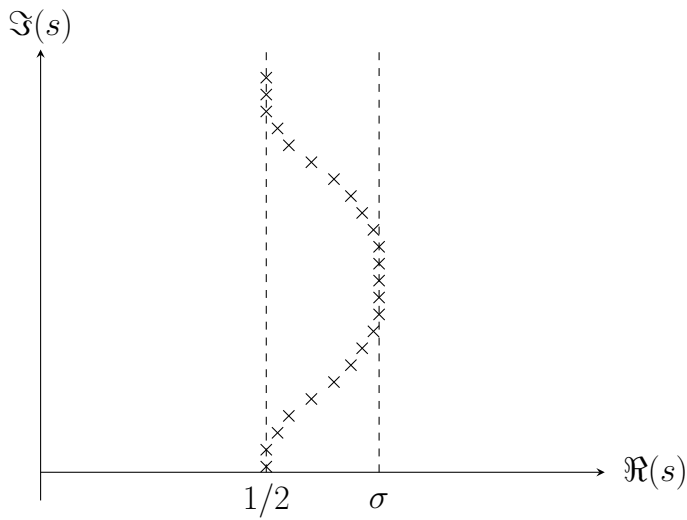
\begin{figure}[h!!!]
\begin{tikzpicture}[scale=0.75,transform shape]
\draw [->,>=stealth] (2*2,0)--(2*7,0);
\draw [->,>=stealth] (2*2,-0.5)--(2*2,7.5);
\draw [dashed] (2*4,0)--(2*4,7.5);
\draw [dashed] (2*5,0)--(2*5,7.5);
\node at (2*7.7-0.5,0) {\Large{$\Re(s)$}};
\node at (2*2,7.9) {\Large{$\Im(s)$}};
\node at (2*4,-0.5) {\Large{1/2}};
\node at (2*5,-0.5) {\Large{$\sigma$}};

\draw (2*3.95,5.7+1.2)--(2*4.05,5.9+1.2);
\draw (2*3.95,5.9+1.2)--(2*4.05,5.7+1.2);

\draw (2*3.95,5.4+1.2)--(2*4.05,5.6+1.2);
\draw (2*3.95,5.6+1.2)--(2*4.05,5.4+1.2);

\draw (2*3.95,5.1+1.2)--(2*4.05,5.3+1.2);
\draw (2*3.95,5.3+1.2)--(2*4.05,5.1+1.2);

\draw (2*4.05,4.8+1.2)--(2*4.15,5.0+1.2);
\draw (2*4.05,5.0+1.2)--(2*4.15,4.8+1.2);

\draw (2*4.15,4.5+1.2)--(2*4.25,4.7+1.2);
\draw (2*4.15,4.7+1.2)--(2*4.25,4.5+1.2);

\draw (2*4.35,4.2+1.2)--(2*4.45,4.4+1.2);
\draw (2*4.35,4.4+1.2)--(2*4.45,4.2+1.2);

\draw (2*4.55,3.9+1.2)--(2*4.65,4.1+1.2);
\draw (2*4.55,4.1+1.2)--(2*4.65,3.9+1.2);

\draw (2*4.7,3.6+1.2)--(2*4.8,3.8+1.2);
\draw (2*4.7,3.8+1.2)--(2*4.8,3.6+1.2);

\draw (2*4.8,3.3+1.2)--(2*4.9,3.5+1.2);
\draw (2*4.8,3.5+1.2)--(2*4.9,3.3+1.2);

\draw (2*4.9,3.0+1.2)--(2*5.0,3.2+1.2);
\draw (2*4.9,3.2+1.2)--(2*5.0,3.0+1.2);

\draw (2*4.95,2.7+1.2)--(2*5.05,2.9+1.2);
\draw (2*4.95,2.9+1.2)--(2*5.05,2.7+1.2);

\draw (2*4.95,3.6)--(2*5.05,3.8);
\draw (2*4.95,3.8)--(2*5.05,3.6);

\draw (2*4.95,3.3)--(2*5.05,3.5);
\draw (2*4.95,3.5)--(2*5.05,3.3);

\draw (2*4.95,3.0)--(2*5.05,3.2);
\draw (2*4.95,3.2)--(2*5.05,3.0);

\draw (2*4.95,2.7)--(2*5.05,2.9);
\draw (2*4.95,2.9)--(2*5.05,2.7);

\draw (2*4.9,2.4)--(2*5.0,2.6);
\draw (2*4.9,2.6)--(2*5.0,2.4);

\draw (2*4.8,2.1)--(2*4.9,2.3);
\draw (2*4.8,2.3)--(2*4.9,2.1);

\draw (2*4.7,1.8)--(2*4.8,2.0);
\draw (2*4.7,2.0)--(2*4.8,1.8);

\draw (2*4.55,1.5)--(2*4.65,1.7);
\draw (2*4.55,1.7)--(2*4.65,1.5);

\draw (2*4.35,1.2)--(2*4.45,1.4);
\draw (2*4.35,1.4)--(2*4.45,1.2);

\draw (2*4.15,0.9)--(2*4.25,1.1);
\draw (2*4.15,1.1)--(2*4.25,0.9);

\draw (2*4.05,0.6)--(2*4.15,0.8);
\draw (2*4.05,0.8)--(2*4.15,0.6);

\draw (2*3.95,0.3)--(2*4.05,0.5);
\draw (2*3.95,0.5)--(2*4.05,0.3);

\draw (2*3.95,0.0)--(2*4.05,0.2);
\draw (2*3.95,0.2)--(2*4.05,0.0);
\end{tikzpicture}
\caption{A bow of zeros, marked by crosses} \label{fig:bow}
\end{figure}
\end{center}

Locally (at a scale $T^{o(1)}$) every zero (except for the few zeros near the ends of the bow) looks like a zero in the middle of a vertical arithmetic progression, since the real parts are almost constant at this scale. Therefore it seems impossible to make the ``local'' power sum-type arguments of Section \ref{sec:HalfIsolated} work, since a variation of Remark \ref{rmk:Turan} (3) using Poisson summation would show that the corresponding sum over zeros would be uniformly small when trying to detect any zero in the middle of a bow. 

On the other hand, since the bow is based on the $\frac{1}{2}$-line and only has length $T^\epsilon$, it does not seem possible to use ``global'' clustering arguments to advantage either. Finding a short zero detector of a zero at the end of the bow would be of no use, since the zero would lie on the $\frac{1}{2}$ line and so would not correspond to an unusually large value of the zero detecting polynomial. Since the length of the cluster is only of length $T^\epsilon$, we could only hope to improve bounds by $T^{O(\epsilon)}$, which would be negligible as $\epsilon\rightarrow 0$. Therefore none of the new ideas in this paper appear to help with detecting these bows.

Finally, classical zero detecting arguments appear to offer no help in detecting bows of zeros. It appears entirely consistent with the Ingham-Huxley density estimate \eqref{eq:ZeroDensity} (and the underlying methods) that there are $\gtrapprox T^{3/5}$ bows of zeros \eqref{eq:BowOfZeros} at height $T_0\asymp T$. If the vertical gaps between zeros were significantly smaller than the average gap, then one might hope to use zero-repulsion techniques to rule out the possibility of there being so many nearby zeros away from the $\frac{1}{2}$-line, but this argument appears break down unless $c$ is sufficiently small. When the clusters have small length, quantities such as $N^*(\sigma,T)$ from \cite{HeathBrownII,HeathBrownIII} which estimate the approximate additive energy also appear to be of no help.

For these reasons the rigidity on the real parts of zeros imposed by Hypothesis $\mathcal{F}$ seems difficult to dispense with. If one could show that the only configurations of points where the corresponding sums are uniformly small come from a small number of vertical arithmetic progressions, then detecting bows of zeros would be essentially the only obstacle to an unconditional result.

\begin{remark}
The configuration of ``bows'' of zeros has some similarities to problematic configurations of zeros which obstruct improvements to the least quadratic non-residue (see \cite[Appendix]{Diamond}). We thank Roger Heath-Brown for pointing this out to us.
\end{remark}

%
%
%
%
\appendix
%
%
%
%
\section{The fifth moment of the zeta function}

For the convenience of the reader we restate Theorem \ref{thm:FifthMoment} as the following proposition.
%
%
%
%
\begin{proposition}\label{prop:FifthMoment}
Suppose that all of the non-trivial zeros of the Riemann zeta function lie on finitely many vertical lines. Suppose further that on each vertical line with real part $> \frac{1}{2}$ the zeros only occur in clusters of size $\leq T^{\varepsilon_0}$. Then we have
\begin{align*}
\int_0^T \Bigl| \zeta \Bigl(\frac{1}{2}+it \Bigr) \Bigr|^{5}dt \ll_\epsilon T^{1+O(\varepsilon_0)+\epsilon}.
\end{align*}
\end{proposition}
%
%
%
%
The key input (beyond our short zero-detecting Dirichlet polynomials), will be Watt's mean value theorem.
%
%
%
%
\begin{lemma}[Watt's mean value theorem]\label{lem:WattMVT}
\begin{align*}
\int_{-T}^T \Bigl| \zeta \Bigl(\frac{1}{2}+it \Bigr) \Bigr|^4 \Bigl|\sum_{m \leq M} a_m m^{-it} \Bigr|^2dt \ll_\epsilon M T^{1+\epsilon} \Bigl(1 + \frac{M^2}{T^{1/2}} \Bigr) \max_m |a_m|^2.
\end{align*}
\end{lemma}
%
%
%
%
\begin{proof}
See \cite[Theorem 1]{Wat1995}.
\end{proof}
%
%
%
%
\begin{lemma}[Approximate functional equation]\label{lem:AppFuncEq}
There is a smooth function $w_2(x) = w_2(x;T)$ supported on $[1/2,2]$ with $w_2^{(j)}(x) \ll_j 1$ such that
\[
\Bigl|\zeta\Big(\frac{1}{2}+it\Bigr)\Bigr|\ll 1 + \log T\sum_{N=2^j\le T^{1/2+o(1)}}\int_{|u|\leq \log T} \Bigl|\sum_{n}\frac{w_2(n/N)}{n^{1/2+it+iu}}\Bigr| du.
\]
\end{lemma}
%
%
%
%
\begin{proof}
We apply \cite[Theorem 12.2]{Iwa2014} with $\frac{T^{1/2}}{\sqrt{2\pi}}\leq x \leq 2\frac{T^{1/2}}{\sqrt{2\pi}}$, then multiply both sides by $\frac{1}{x}$ and integrate over the range of $x$. After a change of variables we derive
\begin{align*}
\Bigr|\zeta\Bigl(\frac{1}{2}+it\Bigr)\Bigl| \ll \int_{T^{1/2}/2\sqrt{2\pi}}^{2T^{1/2}/\sqrt{2\pi}} \left|\sum_{n\geq 1} \frac{1}{n^{1/2+it}} \exp (-n/x) \right| \frac{dx}{x} + 1.
\end{align*}
We use Lemma \ref{lem:Partition} to obtain
\begin{align*}
\Bigr|\zeta\Bigl(\frac{1}{2}+it\Bigr)\Bigl| \ll \sum_{N = 2^j \leq T^{1/2}(\log T)^2}\int_{T^{1/2}/2\sqrt{2\pi}}^{2T^{1/2}/\sqrt{2\pi}} \left|\sum_{n\geq 1} \frac{w_0(n/N)}{n^{1/2+it}} \exp (-n/x) \right| \frac{dx}{x} + 1.
\end{align*}
We then perform Mellin inversion on the factor $\exp(-n/x)$ to deduce
\begin{align*}
\left|\sum_{n\geq 1} \frac{w_0(n/N)}{n^{1/2+it}} \exp (-n/x) \right| &\ll \int_{-\infty}^\infty \left|\Gamma \left(\frac{1}{\log T} + iu \right) \right| \left|\sum_{n\geq 1} \frac{w_0(n/N) (n/N)^{-1/\log T}}{n^{1/2+it+iu}} \right| du.
\end{align*}
By Stirling's formula we can truncate the integral to $|u| \leq \log T$ at the cost of an admissible error, and then we use a trivial bound for the Gamma function.
\end{proof}
%
%
%
%
We now turn to the proof of Proposition \ref{prop:FifthMoment}. By dyadic subdivision it suffices to show
\begin{align*}
\int_T^{2T} \Bigl| \zeta \Bigl(\frac{1}{2}+it \Bigr) \Bigr|^{5}dt \ll_\epsilon T^{1+O(\varepsilon_0)+\epsilon}
\end{align*}
for all $T$ sufficiently large. We decompose the fifth moment as $5=4+1$, and represent one factor of $\zeta(\frac{1}{2}+it)$ with the approximate functional equation. We then manipulate the resulting Dirichlet polynomials so that the moment integral can be handled by a combination of Watt's mean value theorem and short zero-detecting polynomials.

By Lemma \ref{lem:AppFuncEq}, it suffices to show
\begin{align*}
\int_T^{2T} \Bigl| \zeta \Bigl(\frac{1}{2}+it \Bigr) \Bigr|^{4} \Bigl| \sum_n \frac{ \omega(n/N) }{n^{1/2+it}}\Bigr| dt \ll_\varepsilon T^{1+O(\varepsilon_0)+\varepsilon}
\end{align*}
for $100 \leq N \ll T^{1/2+o(1)}$ (one can immediately use a trivial bound and a fourth moment bound if $N< 100$). Here $\omega(x) = w_2(x) x^{-iu}$ for some fixed real $|u| \leq \log T$. Note that $\omega^{(j)}(x) \ll_j (\log T)^j$. We shall write $\Omega(s)$ for the Mellin transform of $\omega$.

We now split up the Dirichlet polynomial according to the prime factors of $n$. Write $n=n_1n_2$ with $P^+(n_1)\le X^\epsilon<P^-(n_2)$. We split our attention depending on whether $n_1>T^{1/4}$ or not. Thus
\[
\sum_{n}\frac{\omega(n/N)}{n^s}=\sum_{\substack{ n_1\le T^{1/4}\\ P^+(n_1)\le T^\epsilon}}\sum_{P^-(n_2)>T^\epsilon}\frac{\omega(n_1n_2/N)}{(n_1n_2)^s}+\sum_{\substack{ n_1>T^{1/4}\\ P^+(n_1)\le T^\epsilon}}\sum_{P^-(n_2)>T^\epsilon}\frac{\omega(n_1n_2/N)}{(n_1n_2)^s}.
\]
Therefore, it suffices to show that for each $N<T^{1/2+o(1)}$ that
\begin{align*}
\int_T^{2T} \Bigl| \zeta \Bigl(\frac{1}{2}+it \Bigr) \Bigr|^{4}\Bigl|\sum_{\substack{ n_1\le T^{1/4}\\ P^+(n_1)\le T^\epsilon}}\sum_{P^-(n_2)>T^\epsilon}\frac{\omega(n_1n_2/N)}{(n_1n_2)^s} \Bigr| dt& \ll_\epsilon T^{1+O(\varepsilon_0)+\epsilon},\\
\int_T^{2T} \Bigl| \zeta \Bigl(\frac{1}{2}+it \Bigr) \Bigr|^{4}\Bigl|\sum_{\substack{ n_1>T^{1/4}\\ P^+(n_1)\le T^\epsilon}}\sum_{P^-(n_2)>T^\epsilon}\frac{\omega(n_1n_2/N)}{(n_1n_2)^s} \Bigr| dt &\ll_\epsilon T^{1+\epsilon}.
\end{align*}
We first deal with the terms when $n_1>T^{1/4}$, which can be handled unconditionally by factoring the Dirichlet polynomials and using Watt's mean value theorem.
%
%
%
%
\begin{lemma}
For $100\leq N\le T^{1/2+o(1)}$, we have
\[
\int_T^{2T} \Bigl| \zeta \Bigl(\frac{1}{2}+it \Bigr) \Bigr|^{4}\Bigl|\sum_{\substack{ n_1>T^{1/4}\\ P^+(n_1)\le T^\epsilon}}\sum_{P^-(n_2)>T^\epsilon}\frac{\omega(n_1n_2/N)}{(n_1n_2)^s} \Bigr| dt \ll_\varepsilon T^{1+\varepsilon}.
\]
\end{lemma}
%
%
%
%
\begin{proof}
We write the prime factorization of $n_1n_2$ as $n_1n_2=p_1^{e_1}\cdots p_r^{e_r}$ with $p_1 < p_2< \dots < p_r$. Then there is a unique $j$ such that $p_1^{e_1}\cdots p_{j-1}^{e_{j-1}}\le T^{1/4}<p_1^{e_1}\cdots p_{j}^{e_{j}}$ since $n_1>T^{1/4}$, and moreover $p_j\leq T^\epsilon$ since $P^+(n_1)\le T^\epsilon$. Thus we may rewrite $n_1n_2$ as $n_1n_2=n p^e m$ for some $n,p,e,m$ with $P^+(n)<p<P^-(m)$ and $n\le T^{1/4}<p^en$. Thus we see that
\[
\sum_{\substack{ n_1>T^{1/4}\\ P^+(n_1)\le T^\epsilon}}\sum_{P^-(n_2)>T^\epsilon}\frac{\omega(n_1n_2/N)}{(n_1n_2)^s}=\sum_{\substack{n,m,p^e\\ P^+(n)<p<P^-(m)\\ n\le T^{1/4}<n p^e\\ p\leq T^\epsilon}}\frac{\omega(n p^e m/N)}{(n p^e m)^s}.
\]
We now wish to separate the variables $n,m$ and $p^e$. We do this by using Mellin inversion and Perron's formula. For integers $n,p^e,m\leq T^{1/2+\epsilon}$, we see that 
\[
\Bigl|\log\Bigl(\frac{p+1/2}{P^+(n)}\Bigr)\Bigr|,\Bigl|\log\Bigl(\frac{P^-(m)}{p+1/2}\Bigr)\Bigr|,\Bigl|\log\Bigl(\frac{\lfloor T^{1/4}\rfloor +1/2}{n p^e}\Bigr)\Bigr|\gg \frac{1}{T^{1/2+\epsilon}}.
\]
Thus, setting $\delta = \frac{1}{\log T}$ and using the rapid decay of $\Omega(s)$ we find
\begin{align*}
&\sum_{\substack{n,p^e,m\\ P^+(n)<p<P^-(m)\\ n\le T^{1/4}<np^e\\ p<T^\epsilon}}\frac{\omega(np^em/N)}{(np^em)^{s}} = \frac{1}{(2\pi)^4}\int\limits_{\substack{|u_j|\le T^{10} \\ 1\leq j\leq 4}} \frac{A(s)B(s)C(s)\Omega(iu_4)N^{iu_4}}{(\lfloor T^{1/4}\rfloor +1/2)^{\delta + iu_3}}\prod_{j=1}^3 (\delta + iu_j)^{-1}d\textbf{u}+O\Bigl(\frac{1}{T^2}\Bigr),
\end{align*}
where we have defined (suppressing the dependency on $u_1,u_2,u_3,u_4$)
\begin{align*}
A(s)&:=\sum_{n\leq T^{1/4}}\frac{a_{n}}{n^s},\quad &B(s)&:=\sum_{e\le \log{T}}\sum_{p<T^{\epsilon}}\frac{b_{p,e}}{p^{es}},\quad &C(s)&:=\sum_{m<T^{1/4+\epsilon} }\frac{c_{m}}{m^s},\\
a_{n}&:=P^+(n)^{-\delta-i u_1}n^{\delta+ iu_3-iu_4},\quad &b_{p,e}&:=(p+1/2)^{iu_1-iu_2}p^{e(\delta+ iu_3-iu_4)},\quad &c_{m}&:=P^-(m)^{\delta+iu_2}m^{-iu_4}.
\end{align*}
The $O(T^{-2})$ term is clearly negligible, and so can be ignored. By the triangle inequality (and the fact the integral of $|\Omega(iu_4)|\prod_{j=1}^3 |\delta + iu_j|^{-1}$ over $u_1,u_2,u_3,u_4$ is $T^{o(1)}$), we find that
\begin{align*}
&\int_T^{2T} \Bigl| \zeta \Bigl(\frac{1}{2}+it \Bigr) \Bigr|^{4}\Bigl|\sum_{\substack{ n_1>T^{1/4}\\ P^+(n_1)\le T^\epsilon}}\sum_{P^-(n_2)>T^\epsilon}\frac{\omega(n_1n_2/N)}{(n_1n_2)^s} \Bigr| dt \\
&\ll T^{o(1)}\sup_{u_1,u_2,u_3,u_4}\int_T^{2T} \Bigl| \zeta \Bigl(\frac{1}{2}+it \Bigr) \Bigr|^{4}\Bigl|A\Bigl(\frac{1}{2}+it\Bigr)B\Bigl(\frac{1}{2}+it\Bigr)C\Bigl(\frac{1}{2}+it\Bigr)\Bigr| dt +O(1).
\end{align*}
We now note that trivially we have the bound
\[
B\Bigl(\frac{1}{2}+it\Bigr)=\sum_{e\le \log{T}}\sum_{p\le T^\epsilon}\frac{b_{p,e}}{p^{e/2+eit}}\ll T^{\epsilon/2},
\]
so the $B$ factor is always acceptably small. Thus, by Cauchy-Schwarz, it suffices to show 
\[
\sup_{u_1,u_2,u_3,u_4}\int_T^{2T} \Bigl| \zeta \Bigl(\frac{1}{2}+it \Bigr) \Bigr|^{4}\Bigl|A\Bigl(\frac{1}{2}+it\Bigr)\Bigr|^{2}\ll_\epsilon T^{1+\epsilon},
\]
and similarly for $C(1/2+it)$ in place of $A(1/2+it)$. But this follows immediately from Lemma \ref{lem:WattMVT} since these Dirichlet polynomials have length at most $T^{1/4+\epsilon}$.
\end{proof}
%
%
%
%
We now consider the terms with $n_1\le T^{1/4}$. Since $n_1n_2 \leq T^{1/2+o(1)}$ we think of $n_2$ as being somewhat large. By factoring $n_2$ we introduce sums over primes, and by Mellin inversion we turn the sums over primes into sums over zeros. This transforms the moment integral into an integral centred on non-trivial zeros. We then use Hypothesis $\mathcal{F}$ and short zero-detecting polynomials to control the integral centred on zeros by sums which can be handled with Watt's mean value theorem.

The next result is the unconditional lemma to reduce to an integral centred on zeros $\rho$.
%
%
%
%
\begin{lemma}[Reduction to sum over zeros]
We have
\begin{align*}
&\int_T^{2T} \left|\zeta \Bigl(\frac{1}{2}+it \Bigr) \right|^{4}\Bigl|\sum_{\substack{ n_1\le T^{1/4}\\ P^+(n_1)\le T^\epsilon}}\sum_{P^-(n_2)>T^\epsilon}\frac{\omega(n_1n_2/N)}{(n_1n_2)^s}\Bigr| dt \ll  T^{1+\epsilon} \\
&+T^{\epsilon}\sup_{\substack{|u|<T^{\epsilon}\\ N_1\ll T^{1/4}}}\sum_{T/2\le |\rho|\le 3T}\Bigl(\frac{T^{1/2}}{N_1}\Bigr)^{\textup{Re}(\rho)-1/2}\int_{|t-\textup{Im}(\rho)|<T^{3\epsilon}}\Bigl|\zeta\Bigl(\frac{1}{2}+it\Bigr)\Bigr|^4 \Bigl|\sum_{\substack{N_1\le n_1<2N_1\\ P^+(n_1)<T^\epsilon\\ n_1<T^{1/4}}}\frac{1}{n_1^{1/2+it+iu}}\Bigr| dt.
\end{align*}
\end{lemma}
%
%
%
%
\begin{proof}
The contribution from $n_2=1$ is $O(T^{1+o(1)})$ by Cauchy-Schwarz and Lemma \ref{lem:WattMVT}, so we may assume $n_2 > 1$. We expand the $n_2$ sum according to the number of distinct prime factors of $n_2$ (which must be at most $1/\epsilon$ since the sum is supported on $n_2\le T$ with $P^-(n_2)\ge T^\epsilon$). By inclusion-exclusion, there are constants $c_{\mathbf{e}}=c_{J,e_1,\dots,e_\ell}\ll_\epsilon 1$ such that
\[
\sum_{\substack{P^-(n_2)>T^\epsilon \\ n_2 > 1}}\frac{\omega(n_1n_2/N)}{(n_1n_2)^s}=\sum_{1\leq J\le \frac{1}{\epsilon}}\sum_{\substack{e_1,\dots,e_J\ge 1}}c_{\mathbf{e}}\sum_{T^\epsilon< p_1,\dots,p_J}\frac{\omega(n_1 p_1^{e_1}\cdots p_J^{e_J}/N)}{(n_1p_1^{e_1}\cdots p_J^{e_J})^s}.
\]
We use Lemma \ref{lem:Partition} to apply a smooth partition of unity $1=\sum_{P=2^r}w_0(p/P)$ to each variable $p_1,\dots,p_J$. We also restrict $n_1$ to lie in dyadic intervals $[N_1,2N_1)$. This gives
\[
\sum_{\substack{n_1\leq T^{1/4}\\ P^+(n_1)\le T^{\epsilon}}}\sum_{T^\epsilon< p_1,\dots,p_J}\frac{\omega(n_1 p_1^{e_1}\cdots p_J^{e_J}/N)}{(p_1^{e_1}\cdots p_J^{e_J})^s}=\sum_{N_1=2^{j_0}\ll T^{1/4}}\sum_{\substack{P_1=2^{r_1},\dots,P_J=2^{r_J}}}S(s),
\]
where
\[
S(s):=\sum_{T^\epsilon< p_1,\dots,p_J}\sum_{\substack{N_1\le n_1<2N_1\\ P^+(n_1)\leq T^\epsilon\\ n_1\leq T^{1/4}}}\frac{\omega(n_1 p_1^{e_1}\cdots p_J^{e_J}/N)w_0(p_1/P_1)\cdots w_0(p_J/P_J)}{(n_1p_1^{e_1}\cdots p_J^{e_J})^s}.
\]
With these restrictions, we see there is no contribution unless $N_1 P_1\cdots P_J\ll N$. We now separate variables. By Mellin inversion and the rapid decay of $\Omega$, we have for $\text{Re}(s)\in[0,1]$
\begin{align*}
S(s)&=\frac{1}{2\pi i}\int_{-i\infty}^{i\infty}\Omega(z)P_1(s+z)\cdots P_J(s+z)N(s+z) dz\\
&=\frac{1}{2\pi} \int_{-T^{\epsilon}}^{T^{\epsilon}}\Omega(iu)P_1(s+iu)\cdots P_J(s+iu)N(s+iu) du+O(T^{-1}),
\end{align*}
where
\[
P_j(s):=\sum_{T^\epsilon<p_j}\frac{w_0(p_j/P_j)}{p_j^{e_js}},\qquad N(s):=\sum_{\substack{N_1\le n_1<2N_1\\ P^+(n_1)\leq T^\epsilon\\ n_1\leq T^{1/4}}}\frac{1}{n_1^s}.
\]
The $O(T^{-1})$ term above makes a negligible contribution, and so can be ignored.

Since there are $O(T^{o(1)})$ choices of $J,\ell,e_1,\dots,e_\ell,N_1,P_1,\dots,P_J$ and $u$ is restricted to an interval of length $T^\epsilon$ and $\Omega(iu)\ll 1$, we see that it suffices to show that for any choice of $J,\ell,e_1,\dots,e_\ell,N_1,P_1,\dots,P_J$ and $u$ under consideration we have
\begin{align*}
&\int_T^{2T}\Bigl|\zeta\Bigl(\frac{1}{2}+it\Bigr)\Bigr|^4\Bigl|P_1\Bigl(\frac{1}{2}+i(t+u)\Bigr)\cdots P_J\Bigl(\frac{1}{2}+i(t+u)\Bigr)N\Bigl(\frac{1}{2}+i(t+u)\Bigr)\Bigr| dt\\
&\ll 
T^{1+\epsilon}+T^{\epsilon}\sum_{T/2\le |\rho|\le 3T}\Bigl(\frac{T^{1/2}}{N_1}\Bigr)^{\text{Re}(\rho)-1/2}\int_{|t-\text{Im}(\rho)|<T^{3\epsilon}}\Bigl|\zeta\Bigl(\frac{1}{2}+it\Bigr)\Bigr|^4 \Bigl|N\Bigl(\frac{1}{2}+it+iu\Bigr)\Bigr| dt.
\end{align*}
If $e_j>1$ then we see that
\[
P_j\Bigl(\frac{1}{2}+it\Bigr)\ll \sum_{p_j \asymp P_j}\frac{1}{p_j}\ll 1.
\]
Thus we restrict our attention to $P_j$ with $e_j=1$. Similarly, we see that if $e_j=1$
\[
P_j\Bigl(\frac{1}{2}+it\Bigr)=\sum_{T^\epsilon<n}\frac{\Lambda(n)w_0(n/P_j)}{n^{1/2+it}\log{n}}+O(1).
\]
To simplify notation we let
\[
\widetilde{w}_0(t)=\frac{w_0(t)}{\log{t}+\log{P_j}},
\]
which is smooth, supported on $[1/2,2]$ and satisfies $|\widetilde{w}_0^{(j)}(t)|\ll_j 1$. We now use Perron's formula to remove the condition $T^\epsilon<n$ (which only intrudes if $P_j \asymp T^\epsilon$). Let $T_0=\lfloor T^\epsilon\rfloor+1/2$. If $P_j \asymp T^\epsilon$ we have
\begin{equation}\label{eq:Pj}
\begin{split}
P_j\Bigl(\frac{1}{2}+it\Bigr)&=\sum_{T^{\epsilon}<n}\frac{\Lambda(n)\widetilde{w}_0(n/P_j)}{n^{1/2+it}}+O(1)\\
&=\frac{1}{2\pi i}\int_{-T^{2\epsilon}}^{T^{2\epsilon}}\Bigl((2P_j)^{iv}-T_0^{iv}\Bigr)\Bigl(\sum_{n}\frac{\Lambda(n)\widetilde{w}_0(n/P_j)}{n^{1/2+i(v+t)}}\Bigr)\frac{dv}{v}+O(1).
\end{split}
\end{equation}
However, using Mellin inversion again and moving the line of integration, we see that
\begin{align*}
\sum_{n}\frac{\Lambda(n)\widetilde{w}_0(n/P_j)}{n^{1/2+i(v+t)}}&=\frac{1}{2\pi i}\int_{2-i\infty}^{2+i\infty} -P_j^{s}\widetilde{W}_0(s)\frac{\zeta'}{\zeta}\Bigl(\frac{1}{2}+i(v+t)+s\Bigr)ds\\
&=P_j^{1/2-i(v+t)}\widetilde{W}_0\Bigl(\frac{1}{2}-i(v+t)\Bigr)-\sum_{\rho}P_j^{\rho-1/2-i(v+t)}\widetilde{W}_0\Bigl(\rho-\frac{1}{2}-i(v+t)\Bigr)\\
&\qquad+\frac{1}{2\pi i}\int_{-2-i\infty}^{-2+i\infty} -P_j^{s}\widetilde{W}_0(s)\frac{\zeta'}{\zeta}\Bigl(\frac{1}{2}+i(v+t)+s\Bigr)ds.
\end{align*}
By the rapid decay of $\widetilde{W}_0$ and the bound $|\zeta'/\zeta(s)|\ll \log{|s|}$ for $\text{Re}(s)=-3/2$, we see that for $|v+t|\gg T^{o(1)}$, where the $o(1)$ term goes to zero sufficiently slowly with $T$, we have
\[
\sum_{n}\frac{\Lambda(n)\widetilde{w}_0(n/P_j)}{n^{1/2+i(v+t)}}\ll \sum_{\substack{\rho=\beta+i\gamma\\ |\gamma-i(v+t)|\le T^{o(1)}}}P_j^{\beta-1/2}+O(1).
\]
Substituting this into \eqref{eq:Pj}, we find that if $P_j \asymp T^\epsilon$ then
\[
P_j\Bigl(\frac{1}{2}+it\Bigr)\ll T^{o(1)}\int_{|v|\leq T^{2\epsilon}} \frac{1}{1+|v|} \sum_{\substack{\rho = \beta+i\gamma\\ |\gamma-(t+v)|\le T^{o(1)}}}P_j^{\beta-1/2}+O(T^{o(1)}).
\]
If $P_j \gg T^\epsilon$ we may obtain the same upper bound by using Lemma \ref{lem:DiscreteToContinuous} and then arguing with Mellin inversion as above. It follows that
\begin{align*}
&P_1\Bigl(\frac{1}{2}+i(t+u)\Bigr)\cdots P_J\Bigl(\frac{1}{2}+i(t+u)\Bigr)  \\
&\ll T^{o(1)}+ T^{o(1)}\int_{\substack{|v_j|\leq T^{2\epsilon} \\ 1\leq j\leq J}} \prod_{j=1}^J \frac{1}{1+|v_j|}\sum_{\substack{\rho_1,\ldots,\rho_J \\ \rho_j = \beta_j + i\gamma_j \\ |\gamma_j - (t+u+v_j)|\leq T^{o(1)}}} \prod_{j=1}^J P_j^{\beta_j - \frac{1}{2}}.
\end{align*}
Since $P_1\cdots P_J\ll N/N_1$ and $N\ll T^{1/2+\epsilon}$, we have
\begin{align*}
\prod_{j=1}^J P_j^{\beta_j - \frac{1}{2}} \ll \left(\frac{T^{1/2+\epsilon}}{N_1} \right)^{\text{max}_j \beta_j - \frac{1}{2}},
\end{align*}
and therefore by symmetry
\begin{align*}
P_1\Bigl(\frac{1}{2}+i(t+u)\Bigr)&\cdots P_J\Bigl(\frac{1}{2}+i(t+u)\Bigr) \ll T^{o(1)} \int_{|v_1| \leq T^{2\epsilon}} \frac{1}{1+|v_1|}\sum_{\substack{\rho_1 = \beta_1 + i\gamma_1 \\ |\gamma - (t+u+v_1)|\leq T^{o(1)}}} \left(\frac{T^{1/2+\epsilon}}{N_1} \right)^{\beta_1 - \frac{1}{2}}  \\
&\times \Big(\int_{|v|\leq T^{2\epsilon}}\frac{1}{1+|v|}\Big(\sum_{\substack{\rho = \beta+i\gamma \\ |\gamma - (t+u+v)|\leq T^{o(1)}}} 1 \Big) dv\Big)^{J-1} dv_1+ T^\epsilon \\
&\ll T^{o(1)} \int_{|v_1| \leq T^{2\epsilon}} \frac{1}{1+|v_1|}\sum_{\substack{\rho_1 = \beta_1 + i\gamma_1 \\ |\gamma - (t+u+v_1)|\leq T^{o(1)}}} \left(\frac{T^{1/2+\epsilon}}{N_1} \right)^{\beta_1 - \frac{1}{2}}dv_1+ T^\epsilon \\
&\ll T^\epsilon \sum_{\substack{\rho = \beta+i\gamma \\ |\gamma - t| \leq T^{3\epsilon}}} \Bigl(\frac{T^{1/2}}{N_1}\Bigr)^{\beta-1/2} + T^\epsilon.
\end{align*}
This in turn gives
 \begin{align*}
&\int_T^{2T}\Bigl|\zeta\Bigl(\frac{1}{2}+it\Bigr)\Bigr|^4\Bigl|P_1\Bigl(\frac{1}{2}+i(t+u)\Bigr)\cdots P_J\Bigl(\frac{1}{2}+i(t+u)\Bigr)N\Bigl(\frac{1}{2}+i(t+u)\Bigr)\Bigr| dt\\
&\ll 
T^{1+\epsilon}+T^{\epsilon}\sum_{T/2\le |\rho|\le 3T}\Bigl(\frac{T^{1/2}}{N_1}\Bigr)^{\text{Re}(\rho)-1/2}\int_{|t-\text{Im}(\rho)|<T^{3\epsilon}}\Bigl|\zeta\Bigl(\frac{1}{2}+it\Bigr)\Bigr|^4 \Bigl|N\Bigl(\frac{1}{2}+it+iu\Bigr)\Bigr| dt,
\end{align*}
 as required.
 \end{proof}
 %
%
%
%
In order to finish the proof of Proposition \ref{prop:FifthMoment} it suffices to prove the following result.
 %
%
%
%
 \begin{lemma}
Assume Hypothesis $\mathcal{F}$ and that every non-trivial zero occurs in a cluster of length $\leq T^{\varepsilon_0}$. Then we have
 \[
 \sup_{\substack{|u|<T^{\epsilon}\\ N_1\ll T^{1/4}}}\sum_{T/2\le |\rho|\le 3T}\Bigl(\frac{T^{1/2}}{N_1}\Bigr)^{\textup{Re}(\rho)-1/2}\int_{|t-\textup{Im}(\rho)|<T^{3\epsilon}}\Bigl|\zeta\Bigl(\frac{1}{2}+it\Bigr)\Bigr|^4 \Bigl|\sum_{\substack{N_1\le n_1<2N_1\\ P^+(n_1)<T^\epsilon\\ n_1<T^{1/4}}}\frac{1}{n_1^{1/2+it+iu}}\Bigr| dt\ll T^{1+O(\varepsilon_0)+\epsilon}.
 \]
 \end{lemma} 
 %
%
%
%
 \begin{proof}
By Hypothesis $\mathcal{F}$ the non-trivial zeros lie on finitely many vertical lines, and so it suffices to consider each of the possibilities for $\text{Re}(\rho)$ separately. Moreover, since we assume that all zeros lie in a cluster of $\leq T^{\varepsilon_0}$ zeros, every zero with real part $\beta_0$ is within $O(T^{\varepsilon_0+o(1)})$ of a half-isolated zero with real part $\beta_0$. Upon redefining $\epsilon$, it suffices to show that for any choice of $\beta_0$
\[
\Bigl(\frac{T^{1/2}}{N_1}\Bigr)^{\beta_0-1/2}\sum_{\substack{T/2<|\rho_0=\beta_0+i\gamma_0|\leq 3T\\ \text{half-isolated} }}\int_{|t-\gamma_0|<T^{\varepsilon_0 + \epsilon}} \Bigl| \zeta \Bigl(\frac{1}{2}+it \Bigr) \Bigr|^{4}\Bigl|N\Bigl(\frac{1}{2}+it+iu\Bigr)\Bigr| dt\ll T^{1+O(\varepsilon_0)+\epsilon}.
\]

We now recall from Proposition \ref{prp:HalfIsolated} that half-isolated zeros have short zero-detecting Dirichlet polynomials. In particular, there is an $R$ with $\exp ((\log \log T)^3) \ll R \ll T^\epsilon$ and a Dirichlet polynomial
\[
R(s)=\sum_{R/2 < n \leq R}\frac{a_n}{n^s}
\]
with $|a_n|\ll \Lambda(n)$ such that $|R(1/2+i\gamma_0)|\gg R^{\beta_0-1/2}(\log T)^{-100}$ whenever $\beta_0+i\gamma_0$ is a half-isolated zero. Thus we have for any positive integer $\kappa \ll (\log T)/(\log \log T)^3$ with $R^\kappa \geq T^\epsilon$ that
\begin{align*}
&\Bigl(\frac{T^{1/2}}{N_1}\Bigr)^{\beta_0-1/2}\sum_{\substack{T/2<|\rho_0|\leq 3T\\ \text{half-isolated} }}\int_{|t-\gamma_0|<T^{\varepsilon_0 + \epsilon}} \Bigl| \zeta \Bigl(\frac{1}{2}+it \Bigr) \Bigr|^{4}\Bigl|N\Bigl(\frac{1}{2}+it+iu\Bigr)\Bigr|\\
&\ll T^{o(1)} \Bigl(\frac{T^{1/2}}{N_1 R^{\kappa}}\Bigr)^{\beta_0-1/2}\int_{T/4}^{4T} \Bigl| \zeta \Bigl(\frac{1}{2}+it \Bigr) \Bigr|^{4}\Bigl|N\Bigl(\frac{1}{2}+it+iu\Bigr)\Bigr| \sum_{\substack{T/2<|\rho_0|\leq 3T\\ \text{half-isolated} \\ |\gamma_0-t| < T^{\varepsilon_0 + \epsilon}}} \left|R \left( \frac{1}{2}+i\gamma_0 \right) \right|^\kappa dt\\
&\ll_\epsilon  T^{o(1)} \Bigl(\frac{T^{1/2}}{N_1 R^{\kappa}}\Bigr)^{\beta_0-1/2}\int_{|v|\leq T^{\varepsilon_0+2\epsilon}}\int_{T/4}^{4T} \Bigl| \zeta \Bigl(\frac{1}{2}+it \Bigr) \Bigr|^{4}\Bigl|N\Bigl(\frac{1}{2}+it+iu\Bigr)\Bigr| \Bigl|R\Bigl(\frac{1}{2}+it+iv\Bigr)\Bigr|^{\kappa}dt dv,
\end{align*}
where the last inequality follows from Lemma \ref{lem:DiscreteToContinuous}. Let $j_1$ be the non-negative integer such that $N_1 R^{j_1}\le T^{1/4}\le N_1 R^{j_1+1}$. Let $j_2$ be the positive integer such that $R^{j_2}\le T^{1/4}\le R^{j_2+1}$. We now choose
\[
\kappa=j_1+j_2=\frac{\log(T^{1/2}/N_1)}{\log{R}}+O(1).
\]
Cauchy-Schwarz then implies that
\[
\int_{T/4}^{4T} \Bigl| \zeta \Bigl(\frac{1}{2}+it \Bigr) \Bigr|^{4}\Bigl|N\Bigl(\frac{1}{2}+it+iu\Bigr)\Bigr| \Bigl|R\Bigl(\frac{1}{2}+it+iv\Bigr)\Bigr|^{\kappa}dt\ll I_1^{1/2}I_2^{1/2},
\]
where
\begin{align*}
I_1&:=\int_{-4T}^{4T} \Bigl| \zeta \Bigl(\frac{1}{2}+it \Bigr) \Bigr|^{4}\Bigl|N\Bigl(\frac{1}{2}+it+iu\Bigr)\Bigr|^{2} \Bigl|R\Bigl(\frac{1}{2}+it+iv\Bigr)\Bigr|^{2j_1}dt,\\
I_2&:=\int_{-4T}^{4T} \Bigl| \zeta \Bigl(\frac{1}{2}+it \Bigr) \Bigr|^{4}\Bigl|R\Bigl(\frac{1}{2}+it+iv\Bigr)\Bigr|^{2j_2}dt.
\end{align*}
Since $N_1 R^{j_1}<T^{1/4}$ and $R^{j_2}<T^{1/4}$ in both cases we are estimating a twisted fourth moment with a Dirichlet polynomial of length at most $T^{1/4}$. Lemma \ref{lem:WattMVT} therefore gives $I_1,I_2\ll T^{1+\epsilon}$. Thus we obtain the bound
\begin{align*}
&\Bigl(\frac{T^{1/2}}{N_1}\Bigr)^{\beta_0-1/2}\sum_{\substack{T/2<|\rho_0=\beta_0+i\gamma_0|\leq 3T\\ \text{half-isolated} }}\int_{|t-\gamma_0|<T^{\varepsilon_0 + \epsilon}} \Bigl| \zeta \Bigl(\frac{1}{2}+it \Bigr) \Bigr|^{4}\Bigl|N\Bigl(\frac{1}{2}+it+iu\Bigr)\Bigr| dt \\
&\ll  \Bigl(\frac{T^{1/2}}{N_1 R^{\kappa}}\Bigr)^{\beta_0-1/2} T^{1+\varepsilon_0+4\epsilon}.
\end{align*}
Recalling that $R^{\kappa}\gg T^{1/2-\epsilon}N_1^{-1}$ , we see that we obtain an overall bound of $T^{1+\varepsilon_0+5\epsilon}$, as required.
\end{proof}
%
%
%
%

\section{Smooth functions}

We need a smooth function with certain standard properties.

\begin{lemma}\label{lem:Partition}
There is a fixed non-negative smooth function $w_0(x)$ bounded by 1 and supported on $[\frac{1}{2},2]$ such that if $x\geq 1$ then
\begin{align*}
1 = \sum_{m=0}^\infty w_0\Bigl(\frac{x}{2^m} \Bigr).
\end{align*}
The function $w_0(x)$ satisfies the uniform derivative estimate
\begin{align*}
|w_0^{(j)}(x)| \ll 1+ \left( \frac{4j^2}{e^2}\right)^j
\end{align*}
for $j\geq 0$. The Mellin transform $W_0(s)$ of $w_0(x)$ satisfies $W_0(0) = \log 2$, and
\begin{align*}
|W_0(\sigma+it)| &\ll 2^{|\sigma|}\exp \Big(- \sqrt{|t|/2} \Big)
\end{align*}
uniformly in $\sigma$ and $t$.
\end{lemma}
\begin{proof}
Define 
\begin{align*}
h(x) &=
\begin{cases}
\exp\left(-\frac{1}{x(1-x)}\right), \ \ \ \ \ \ &x \in (0,1), \\
0, &\text{otherwise},
\end{cases}
\end{align*}
and
\begin{align*}
H(y) = \frac{1}{C}\int_{-\infty}^y h(x)dx,
\end{align*}
where
\begin{align*}
C = \int_0^1 h(x) dx.
\end{align*}
We then set $w_0(x) = H(2x-1) - H(x-1)$, and claim that $w_0$ has the required properties. It is easy to check that $w_0$ is smooth and supported in $[\frac{1}{2},2]$. Checking that
\begin{align*}
\sum_{m=0}^\infty w_0\Bigl(\frac{x}{2^m} \Bigr)=1
\end{align*}
for $x\geq 1$ is an exercise in telescoping series.

We turn our attention to the derivative bounds. For $j\geq 1$ we have
\begin{align*}
w_0^{(j)}(x) &=
\begin{cases}
\frac{2^j}{C}h^{(j-1)}(2x-1),  \ \ \ \ \ \ &\frac{1}{2}\leq x \leq 1, \\
-\frac{1}{C}h^{(j-1)}(x-1), &1 \leq x \leq 2,
\end{cases}
\end{align*}
so
\begin{align*}
|w_0'(x)| &\leq \frac{2 e^{-4}}{C}.
\end{align*}
For $j\geq 2$ we use \cite[Corollary A.2]{Iwa2014} and then Stirling's formula \cite{Her1955} to get
\begin{align*}
|w_0^{(j)}(x)| &\leq \frac{2^j}{C}\sup_x |h^{(j-1)}(x)| \leq \frac{2^j}{C}(j-1)!\left(\frac{2(j-1)}{e} \right)^{j-1} \\
&\leq \frac{2^j}{C}\left(\frac{j}{e} \right)^j\left(\frac{2j}{e} \right)^j \sqrt{2\pi j}\exp \left( \frac{1}{12j}\right) \frac{e}{2j(j-1)} \left(1 - \frac{1}{j}\right)^j \leq \frac{1}{C}\left( \frac{4j^2}{e^2}\right)^j.
\end{align*}

It remains to bound the Mellin transform of $w_0$. By repeated integration by parts we have
\begin{align*}
W_0(s) &= \frac{(-1)^k}{s(s+1) \cdots (s+k-1)}\int_0^\infty w_0^{(k)}(x) x^{s+k-1}dx
\end{align*}
for any integer $k\geq 0$. On the one hand, using this for $k=0$ and the triangle inequality we have the trivial bound
\begin{align*}
|W_0(s)|&\leq 2^{|\sigma|}\int_{1/2}^2 w_0(x) \frac{dx}{x} = 2^{|\sigma|}W_0(0).
\end{align*}
On the other hand, using the bound for $|w_0^{(k)}(x)|$ gives
\begin{align*}
|W_0(s)| &\ll 2^{|\sigma|}\left(\frac{8k^2}{e^2 |t|} \right)^k.
\end{align*}
If $|t|\geq 8$ we set
we set
\begin{align*}
k = \left\lfloor \frac{|t|^{1/2}}{2\sqrt{2}} \right\rfloor,
\end{align*}
so that
\begin{align*}
|W_0(s)| &\ll 2^{|\sigma|}\exp (-2k)\ll 2^{|\sigma|}\exp \left(-\frac{\sqrt{|t|}}{\sqrt{2}} \right).
\end{align*}

It remains to show that $W_0(0) = \log 2$. We have
\begin{align*}
W_0(0) = \int_{1/2}^2 w_0(x) \frac{dx}{x} = \int_{1/2}^1 H(2x-1) \frac{dx}{x} + \int_1^2 (1-H(x-1)) \frac{dx}{x},
\end{align*}
and some changes of variables reveal
\begin{align*}
\int_{1/2}^1 H(2x-1) \frac{dx}{x} &= \int_1^2 H(x-1) \frac{dx}{x}. \qedhere
\end{align*}
\end{proof}

\section{Zero detecting setup}\label{sec:ZeroDetection}
\begin{proof}[Proof of Lemma \ref{lmm:TypeIIIZeros}]
The technique is essentially that of \cite[Chapter 12]{Mon1971}, but with some small technical refinements. Let $\rho = \beta+i\gamma$ be a non-trivial zero of the Riemann zeta function with $\gamma \in [T,2T]$ and $T$ large. We assume $\beta \geq \frac{1}{2} + \frac{1}{\log T}$. For convenience we write
\begin{align*}
\mathcal{M} &= 2T^{1/100},\\
M(s) &= \sum_{m \leq \mathcal{M}} \frac{\mu(m)}{m^s},\\
Y&=T^{1/2}.
\end{align*}
Now consider the Dirichlet series
\begin{align*}
I(z) := \sum_{n\geq 1} \frac{a(n)}{n^z}\exp \Bigl( -\frac{n}{Y}\Bigr),
\end{align*}
where
\begin{align*}
a(n) := \sum_{\substack{m \mid n \\ m \leq \mathcal{M}}} \mu(m).
\end{align*}
Observe that $a(1)=1, a(n)=0$ for $2 \leq n\leq \mathfrak{M}$, and $|a(n)|\leq \tau(n)$. Hence, separating $n=1$ and using the rapid decay of $\exp(-n/Y)$ for $n>Y(\log{T})^2/2$
\begin{align*}
I(z) = \exp \Bigl( -\frac{1}{Y}\Bigr)+ \sum_{\mathcal{M}<n<Y(\log{T})^2/2} \frac{a(n)}{n^z}\exp \Bigl( -\frac{n}{Y}\Bigr)+O\Bigl(\frac{1}{T}\Bigr).
\end{align*}
Splitting into dyadic blocks,
\[
I(z)=\exp \Bigl( -\frac{1}{Y}\Bigr)+\sum_{\substack{N=2^j\\ \mathcal{M}/2<N<Y(\log{T})^2}}\sum_{n \sim N}\frac{a(n)}{n^z}\exp \Bigl( -\frac{n}{Y}\Bigr) + O\Bigl(\frac{1}{T}\Bigr).
\]
On the other hand, by using Mellin inversion to rewrite $\exp (-n/Y)$, we see that
\begin{align*}
I(z) = \frac{1}{2\pi i} \int_{(2)} Y^s \Gamma(s) M(z+s)\zeta(z+s) ds.
\end{align*}
We are interested in this expression when $z=\rho=\beta+i\gamma$. We move the line of integration to $\text{Re}(s) = -\beta + \frac{1}{2}$, in the process picking up only a contribution from the pole at $s = 1-\rho$, since the simple pole of the gamma function at $s=0$ is cancelled out by the zero $\rho$ of zeta. Therefore
\begin{align*}
I(\rho) &= Y^{1-\rho} \Gamma(1-\rho)M(1) + \frac{1}{2\pi i} \int_{(-\beta + 1/2)} Y^s \Gamma(s) M(\rho+s)\zeta(\rho+s) ds.
\end{align*}
The rapid decay of the gamma function in vertical strips implies
\begin{align*}
Y^{1-\rho} \Gamma(1-\rho)M(1) = O(T^{-1}),
\end{align*}
and since $\exp(-1/Y) = 1+O(Y^{-1})$ we find by comparing our two expressions for $I(\rho)$ that
\begin{align*}
1+O\Bigl(\frac{1}{T^{1/2}}\Bigr) +\sum_{\substack{N=2^j\\ \mathcal{M}/2<N<Y(\log{T})^2}}\sum_{n\sim N}\frac{a(n)}{n^\rho}\exp \Bigl( -\frac{n}{Y}\Bigr)= \frac{1}{2\pi i} \int_{(-\beta + 1/2)} Y^s \Gamma(s) M(\rho+s)\zeta(\rho+s) ds.
\end{align*}
At least one of the following is true:
\begin{align*}
\Big|\sum_{n\sim N}\frac{a(n)}{n^\rho}\exp \Bigl( -\frac{n}{Y}\Bigr)\Big|\geq \frac{1}{3\log{T}},\qquad\text{for some $N=2^j\in [\mathcal{M}/2,T^{1/2}(\log{T})^2]$}
\end{align*}
or
\begin{align*}
\Bigl|\frac{1}{2\pi i} \int_{(-\beta + 1/2)} Y^s \Gamma(s) M(\rho+s)\zeta(\rho+s) ds\Bigr| &\geq \frac{1}{3}. \qedhere
\end{align*}
\end{proof}

\begin{proof}[Proof of Lemma \ref{lem:TypeIIZeroBound}]
We may assume that $T$ is sufficiently large and $\sigma\ge 1/2+1/\log{T}$ since otherwise the result is trivial. We change variables in \eqref{eq:Type II zero condition} and use the triangle inequality to see that any zero $\rho=\beta+i\gamma$ counted by $R_{II}(\sigma,T)$ satisfies
\begin{align*}
\frac{1}{3}\leq \frac{T^{1/4-\beta/2}}{2\pi}\int_{-\infty}^\infty|\Gamma(-\tfrac{1}{2}+\beta+iu)|\Bigl|M\Bigl( \tfrac{1}{2}+i\gamma+iu\Bigr)\zeta\Bigl(\tfrac{1}{2}+i\gamma+iu\Bigr) \Bigr|du.
\end{align*}
We may truncate the integral to $|u|\leq (\log T)^2$, say, using the rapid decay of the gamma function. Since $\beta\ge \sigma\geq1/2+1/\log{T}$, we have
\begin{align*}
|\Gamma(-\tfrac{1}{2}+\beta+iu)| \leq \Gamma(-\tfrac{1}{2}+\beta) \ll \log T.
\end{align*}
Thus
\begin{align*}
\frac{T^{\sigma/2-1/4}}{\log T} \ll \int_{|u|\leq (\log T)^2}\Bigl|M\Bigl( \frac{1}{2}+i\gamma+iu\Bigr)\zeta\Bigl(\frac{1}{2}+i\gamma+iu\Bigr) \Bigr|du.
\end{align*}
Taking fourth powers and using H\"older's inequality yields
\begin{align*}
\frac{T^{2\sigma-1}}{(\log T)^{10}} \ll \int_{|u|\leq (\log T)^2}\Bigl|M\Bigl( \frac{1}{2}+i\gamma+iu\Bigr)\zeta\Bigl(\frac{1}{2}+i\gamma+iu\Bigr) \Bigr|^4 du.
\end{align*}
We may choose a set $\mathcal{T}$ of $(\log T)^3$-separated Type II zeros with imaginary parts in $[T,2T]$ such that
\begin{align*}
R_{II}(\sigma,T) \ll (\log T)^{4}|\mathcal{T}|,
\end{align*}
and therefore
\begin{align*}
R_{II}(\sigma,T) &\ll \frac{(\log T)^{14}}{T^{2\sigma-1}}\sum_{\beta+i\gamma \in \mathcal{T}} \int_{|u|\leq (\log T)^2}\Bigl|M\Bigl( \frac{1}{2}+i\gamma+iu\Bigr)\zeta\Bigl(\frac{1}{2}+i\gamma+iu\Bigr) \Bigr|^4 du \\
&\ll \frac{(\log T)^{17}}{T^{2\sigma-1}} \int_{T/2}^{3T} \Bigl|M\Bigl( \frac{1}{2}+it\Bigr)\zeta\Bigl(\frac{1}{2}+it\Bigr) \Bigr|^4dt.
\end{align*}
By bounds for the twisted fourth moment of the zeta function (see \cite[Theorem 1.1]{HY2010}, \cite[Theorem 1.2]{BBLR2020}, or \cite[Proposition 5.1]{HRS2019}) we have
\begin{align*}
\int_{T/2}^{3T} \Bigl|M\Bigl( \frac{1}{2}+it\Bigr)\zeta\Bigl(\frac{1}{2}+it\Bigr) \Bigr|^4dt \ll T (\log T)^{O(1)},
\end{align*}
and therefore
\begin{align*}
R_{II}(\sigma,T) &\ll T^{2(1-\sigma)} (\log T)^{O(1)}. \qedhere
\end{align*}
\end{proof}

\bibliographystyle{plain}
\bibliography{refs3}

\end{document}